\documentclass[11pt]{article}
\usepackage{amssymb}
\usepackage{graphicx}
\usepackage{xcolor} 
\usepackage{tensor}
\usepackage{fullpage} 
\usepackage{amsmath}
\usepackage{amsthm}
\usepackage{mathrsfs}
\usepackage{verbatim}
\usepackage{hyperref}
\usepackage{cite}
\usepackage{array} 
\usepackage{bbm}

\definecolor{darkgreen}{rgb}{0.1,0.6,0.1}

\usepackage{enumitem}
\setlist[enumerate]{leftmargin=1.2em}
\setlist[itemize]{leftmargin=1.2em}
\usepackage{seqsplit,mathtools}

\newtheorem{theorem}{Theorem}[section]
\newtheorem{corollary}[theorem]{Corollary}
\newtheorem{lemma}[theorem]{Lemma}

\newtheorem{proposition}[theorem]{Proposition}

\theoremstyle{definition}

\theoremstyle{remark}
\newtheorem{remark}[theorem]{Remark}

\numberwithin{equation}{section}

\title{Wellposedness and singularity formation beyond the Yudovich class}
\author{Tarek M. Elgindi, Ryan W. Murray, Ayman R. Said}
\date{\ }

\begin{document}

\maketitle
\begin{abstract}
      
       We introduce a local-in-time existence and uniqueness class for solutions to the 2d Euler equation with unbounded vorticity. Furthermore, we show that solutions belonging to this class can develop stronger singularities in finite time, meaning that they experience finite time blow up and exit the wellposedness class. Such solutions may be continued as weak solutions (potentially non-uniquely) after the singularity. While the general dynamics of 2d Euler solutions beyond the Yudovich class will certainly not be so tame, studying such solutions gives a way to study singular phenomena in a more controlled setting. 
\end{abstract}
\section{Introduction}

The problem of singularity propagation in incompressible fluids is an old and fundamental problem. While the theory of global wellposedness for the 2d Euler equation in sub-critical and critical regimes is now very well understood, much less is known about the long time behaviour, formation and persistence of singular structures. In this paper we will work with $2d$ inviscid flows which solve the vorticity form of Euler's equation, namely
\begin{equation}
    \label{eq: 2dEuler1} \partial_t \omega + u\cdot\nabla \omega=0,
\end{equation}
\begin{equation}\label{eq: 2dEuler2}
    u=\nabla^\perp \psi \text{ and }\Delta\psi=\omega.
\end{equation}
Here we adopt the standard notation $v^\perp=(-v_2, v_1)$ for $v=(v_1,v_2)\in\mathbb{R}^2.$ For this equation, the persistence of singular structures is only rigorously treated in a few special cases. Some of the most notable of these treat the critical regime, such as the bounded vorticity case of vortex patches \cite{chemin1993persistance,bertozzi1993global,EJ1,elgindi2019singular,elgindi2020singular} and the unbounded case in the so-called Yudovich class \cite{yudovich1995uniqueness,serfati1994pertes,drivas2022propagation}, where it is shown that the singularity persists and moves through the fluid under the effect of its own flow. In both of these cases with critical singularities, the Euler equation is known to be globally wellposed. Outside of these cases, very little is known except for the pioneering work of Elling \cite{elling2013algebraic} and Vishik \cite{Vishik1,Vishik2} where they construct non-trivial self-similar trajectories (unforced for the first author and forced for the second) of solutions near algebraically singular steady states or for self-similar supercritical algebraic singularities, see also \cite{albritton2021instability,shao2023self}.

The purpose of this work is to advance the
wellposedness theory of the 2d Euler equations into new regimes through the study of a novel class of algebraic singularities embedded in $L^p(\mathbb{R}^2)$ for $p$ finite, without assuming self-similarity. 
\newpage
Our main Theorems can be stated informally as
\begin{theorem}
Consider solutions $\omega(r,\theta,t)$ to \eqref{eq: 2dEuler1}-\eqref{eq: 2dEuler2} that are odd and $\frac{2\pi}{m}$ periodic in $\theta$ with $m\geq 3$ and singularity not worse than $(r+\theta)^{-\alpha}$ for $\theta\in(0,\frac{\pi}{m})$ for some $0<\alpha<1.$
\begin{itemize}
    \item The class of such solutions  is a local existence, uniqueness, and continuous dependence class for the 2d Euler equation. 
    \item Solutions may exit this class in finite time in the sense that the singularity becomes worse than $(r+\theta)^{-\alpha}$ at some finite time. 
    \item As a result of the singularity formation, there exist Euler solutions that are initially arbitrarily close in $L^p_{loc}(\mathbb{R}^2)$ but separate to order 1 at some fixed finite-time. 
\end{itemize}
\end{theorem} 
We note that velocity fields generated from vorticities bearing this type of algebraic singularity are merely in  $C^{1-\alpha}$ (and not better), even while maintaining local existence. This result can be interpreted as a foray into the dynamics of 2d Euler solutions beyond the Yudovich class: while we still keep local existence and uniqueness, such solutions can build up and develop a further singularity in finite time.  The instability constructed seems to indicate that localized $L^p$ solutions \emph{may} be non-unique after the blow-up time.
\subsection{Background and Related Work}

There is a significant body of literature studying the wellposedness, or lack thereof, of the 2d Euler equations in various function classes. We do not give an exhaustive treatment here, referring the reader to the surveys \cite{drivas2022singularity,buckmaster2020convex}, but instead attempt to focus the discussion on literature related to the present work.

\subsubsection*{Existence and Uniqueness Results}
Global existence and uniqueness classes for 2d Euler have been considered in many classical contexts. The first results we are aware of are the results of Lichtenstein \cite{lichtenstein1925einige}, H\"older \cite{holder1933unbeschrankte},
and Wolibner \cite{wolibner1933theoreme}. These authors proved local and global existence and uniqueness for velocity fields in $C^{k,\alpha}(\mathbb{R}^2)$ spaces with $k \in \mathbb{N}, k \geq 1$ and $0 < \alpha < 1$. A subsequent significant advance was the seminal
work of Yudovich \cite{yudovich1963non} where he proved that for a given $\omega_0\in   L^1 \cap L^\infty(\mathbb{R}^2)$, there exists a unique solution
to the 2d Euler equations remaining in $L^1 \cap L^\infty(\mathbb{R}^2)$. This provides a global wellposedness theory in a critical norm.

\indent Thereafter, Yudovich \cite{yudovich1995uniqueness} and independently  Serfati \cite{serfati1994pertes}, extended the global existence and uniqueness result for $L^1 \cap L^\infty(\mathbb{R}^2)$ vorticity to
allow for mildly unbounded vorticity. More specifically, for each smooth increasing function $\Theta: \mathbb{R}_+\to \mathbb{R}_+$, Yudovich defined the spaces:
\[
Y_\Theta=\left\{ f\in \bigcap_{p\geq 1}L^p \left(\mathbb{R}^2\right), \exists M>0, \ \left\Vert f \right\Vert_{L^p}\leq M \Theta(p), \ p\geq 1   \right\}
\]
and proved that if
\[
\int^{+\infty}_{1}\frac{1}{p \Theta(p)}dp=+\infty,
\]
then $Y_\Theta$ is a global existence and uniqueness class for 2d Euler. The simplest unbounded example of such data would be the case where $\Theta(p) = \ln(1+p)$, which subsequently admits singularities which locally scale like $\ln\left(\ln\left(\frac{1}{\left|x\right|}\right)\right)$. It is, however, important to remark that functions which are even slightly more singular do not belong to Yudovich's uniqueness class. Examples which are \textit{not} covered by Yudovich's
existence/uniqueness result are:
\[
\ln\left(\frac{1}{\left|x\right|}\right) \text{ or }\ln\left(\ln\left(\frac{1}{\left|x\right|}\right)\right)^2.
\]
There are relatively few other settings where wellposedness classes have been established for singular data. One notable example, first addressed by Vishik \cite{vishik1998hydrodynamics, vishik1999}, treats the case of wellposedness in Besov-type spaces. Another notable example establishes existence and uniqueness in BMO type spaces, see for example the papers of Bernicot-Keraani \cite{bernicot2014global}, Hmidi-Keraani \cite{hmidi2008incompressible}, and
Bernicot-Elgindi-Keraani \cite{bernicot2016inviscid}. We remark that neither of these families of results are able to address the existence and uniqueness for data which behave like $\ln\left(\frac{1}{\left|x\right|}\right)$ near the origin or even like $\ln\left(\ln\left(\frac{1}{\left|x\right|}\right)\right)^\alpha$ for $\alpha>1$, let alone the algebraic singularities living in the space $C^\omega_{-\alpha}(\mathbb{R}^2)$ defined in \eqref{eq:C omega alpha}.

\subsection*{Uniqueness Results}
The central idea of much of the uniqueness theory for Euler's equation, and indeed one that underpins Yudovich's work, is to seek for function spaces for the vorticity wherein one can guarantee that the velocity field is Osgood continuous, and hence induces a unique flow map.

Following Yudovich's $L^1$ uniqueness proof, one can show that the class of velocity fields with $\nabla u\in Y_\Theta$ with $\int_1^{+\infty}\frac{1}{\Theta(p)}dp=+\infty$ is a
uniqueness class. Vishik's \cite{vishik1998hydrodynamics} uniqueness class is very similarly defined. In fact, if we define spaces $V_\Theta$ by
\[
V_\Theta=\left\{f\in L^1, \forall N\in \mathbb{N}, \sum_{j=-1}^N \left\Vert \Delta_j f \right\Vert_{L^\infty}\leq M\Theta(N)\right\},
\]
then the set of all velocity fields with $\nabla u \in V_\Theta$ with $\int_1^{+\infty}\frac{1}{\Theta(p)}dp=+\infty$ is a uniqueness
class\footnote{Note that Vishik actually called these spaces $B^\Gamma$ spaces but we have chosen the notation $V_\Theta$ to draw an analogy
between these spaces and Yudovich's}. Here, $\Delta_j$ denotes the Littlewood-Paley projector onto frequencies roughly of size $2^j$. Hence, while functions of the form $\ln\left(\frac{1}{\left\vert x\right\vert}\right)$ do not belong to any known existence \textit{and} uniqueness class, they do belong to a uniqueness class. This means that if one
were able to actually propagate solutions which have a logarithmic singularity, both Yudovich's
and Vishik's uniqueness criteria would allow us to say that the constructed solution is the only
solution which has a logarithmic singularity. 

\subsubsection*{Existence Results: Compactness Arguments}
Though existence/uniqueness and uniqueness results are mostly in function spaces with vorticity close to $L^\infty$, it is possible to prove the existence of weak solutions to the incompressible 2d Euler equations for data with much lower regularity using compactness methods. To our knowledge, the best result in this direction
is J.M. Delort's \cite{delort1991existence} where he proved that existence can be established under the condition that the
velocity field is locally $L^2$ and the vorticity is a positive measure. Earlier important work in this direction is due to DiPerna and Majda \cite{diperna1987concentrations}, who established the existence of weak solutions with $L^p(\mathbb{R}^2)$ vorticity for $p > 1$. A special class of scale invariant and logarithmic spiral \cite{JS} type solutions have also been shown to exist as weak solutions to the 2d Euler equation in $L^p_{loc}(\mathbb{R}^2),p\geq 1$.
\subsubsection*{Ill-posedness, Norm Inflation and Non-uniqueness}
There is also a growing literature regarding the lack of wellposedness for Euler equations. This includes non-solvability of the Euler
equations in certain spaces, non-uniqueness, and discontinuity of the solution map. For results on
non-uniqueness of the Euler equations at very low regularities, we point the reader to the works
of Scheffer \cite{scheffer1993inviscid}, Shnirelman \cite{shnirelman1997nonuniqueness}, De Lellis-Szekhylidi \cite{de2013dissipative}, and Isett \cite{isett2018proof}. These authors all
proved non-uniqueness of solutions to the Euler equations in low regularity spaces, chronologically starting with Scheffer proving non-uniqueness in $L^2_{t,x}$ until recently Isett's result proving non-uniqueness in the class of $C^\alpha_{x,t}$ velocity fields for each $\alpha<\frac{1}{3}$, see also \cite{giri20232d} for $d=2$. We must emphasize, however,
that in all of these works the vorticity is not even a measurable function and is only defined as
a distribution. The only result working in a regime just below the well-known critical wellposedness threshold is Vishik's breakthrough \cite{Vishik1,Vishik2} for the forced Euler equation.  
There are also some results on norm inflation and ill-posedness in critical spaces-
i.e. spaces where the norm scales like the Lipschitz norm of the velocity field. See, for example,
the works of Bourgain and Li \cite{bourgain2015strong,bourgain2015strong2}, Elgindi and Masmoudi \cite{elgindi2017ill},  Elgindi and Jeong \cite{elgindi2020ill}, 
Jeong \cite{jeong2021loss} and C\'ordoba, Mart\'inez-Zoroa and Oza\`nski \cite{cordoba2022instantaneous}. 

\subsubsection*{Existence Results of Singular Trajectories: Self Similarity}
The previous results provide a detailed description of certain types of function spaces wherein we can show the existence and/or uniqueness of solutions, but generally do not attempt to provide any description of the solutions themselves. On the other hand, there is a significant body of literature which seeks to construct solutions with specific qualitative properties, matching observed physical flows, even in situations with singularities. Often these solutions rely upon utilizing physical symmetries to simplify the associated equations. We do not make any attempt to give a general accounting of this literature and the associated study of stability of such solutions, instead referring to the books \cite{majda2002vorticity,saffman1995vortex,drazin2004hydrodynamic}.
Of particular interest in this work is the construction of Elling \cite{elling2013algebraic}, see also \cite{shao2023self}, where in an adapted coordinate system he constructs, using a fixed point argument, an algebraic self-similar solutions to the 2d Euler equation of the form \[\omega(t,x)=\frac{1}{t}\omega(x t^{-\mu})\underset{t \to 0}{\to}r^{-\frac{1}{\mu}}\mathring{\omega}(\theta), \mu \geq \frac{2}{3}\text{ and }\theta \in \mathbb{S}.\]
We notice that, in addition to the self-similar form, Elling's original work requires periodicity in $\theta$, which mirrors the types of assumptions we make. This has notably been relaxed in subsequent work \cite{shao2023self}, with the additional requirement that data does not differ too much from the radially symmetric, algebraically singular steady states.

Another recent work which we mention is due to Vishik \cite{Vishik1,Vishik2}, see also \cite{albritton2021instability}. He first constructs an unstable algebraically singular eigenfunction to the forced 2d Euler equation in self similar coordinates. By then using an unstable manifold construction he deduces a non trivial trajectory reaching $0$ in infinite self similar time and thus in finite time ``physical time". Remarkably, Vishik uses these constructions to the establish non-uniqueness of $L^p$ weak solutions to the forced 2d Euler equation no matter how large $p$ is.

All of these constructions feature non-trivial algebraically singular vorticities, but require the additional structure of self-similarity.

\subsection{Main Results}
We will now state our main results, which are of two types:
\begin{itemize}
    \item Local well-posedness
    \item Finite-time blow up and Instability
\end{itemize}

\subsubsection*{Local Wellposedness}
Henceforth we will work in polar coordinates $x=r\cos(\theta)$ and $y=r\sin(\theta)$ with data which is $m$-fold symmetric and odd in $\theta$, where $m\geq 3$ is a fixed integer. We will thus always work in the sector $\mathcal{A}:=\{r \geq 0, 0 \leq \theta \leq \frac{\pi}{m}\}.$ An example of the type of solutions we wish to study is one with initial data \[\omega_0(r,\theta)=(r+\theta)^{-\alpha},\] for $(r,\theta)\in\mathcal{A}$ and some $\alpha>0$. One way to capture such data is to consider solutions where $(r+\theta)^\alpha\omega_0$ is bounded. Unfortunately, it is not clear how to propagate such information on the solution even if it is satisfied initially. We are able, however, to propagate much stronger information on the singularity, as long as we make stronger assumptions initially. Such a singularity can be captured the analytic-type space $\mathcal{C}^{\omega}_{-\alpha}$ generated by the family of norms
\begin{equation}\label{eq:C omega alpha}
 \left\Vert f\right\Vert_{C^\omega_{-\alpha,\lambda}}=\sum^{+\infty}_{k=0} \frac{\lambda^k}{k!}\left(\left\Vert(r+\theta)^{k+\alpha}\partial^k_r f  \right\Vert_{L^\infty\left(\mathbb{R}_+ \times\left(0,\frac{\pi}{m}\right)\right)}+\left\Vert(r+\theta)^{k+\alpha}\partial^k_\theta f  \right\Vert_{L^\infty\left(\mathbb{R}_+ \times\left(0,\frac{\pi}{m}\right)\right)}\right)<+\infty,
\end{equation}
for some $\lambda>0$. Our main local existence result is a proof of propagation of this norm.
\begin{theorem}\label{thm:loc fix sing prop}
Consider $0\leq \alpha<1$ and $\omega_0\in \mathcal{C}^{\omega}_{-\alpha}$, with $\left\Vert \omega_0 \right\Vert_{C^\omega_{-\alpha,\lambda_0}}<+\infty$ for some $\lambda_0<1$. Then there exists a maximal time of existence $T_*>0$ and a unique $\omega\in C\left([0,T_*],\mathcal{C}^{\omega}_{-\alpha}\right)$ solution of \eqref{eq: 2dEuler1}-\eqref{eq: 2dEuler2}. Moreover there exists a universal constant $C>0$ such that $T_*\geq C\lambda_0 (1-\alpha)^2$ and $\lambda(t)$ decreasing such that
\[\lambda(t)\geq \lambda_0-\frac{C \left\Vert \omega_0 \right\Vert_{C^\omega_{-\alpha,\lambda_0}}}{(1-\alpha)^2}t
\text{ and } \left\Vert \omega(t) \right\Vert_{C^\omega_{-\alpha,\lambda(t)}}\leq C \left\Vert \omega_0\right\Vert_{C^\omega_{-\alpha,\lambda_0}}.
\]
\end{theorem}
\noindent One of the key ingredients is a proof of uniqueness of solutions of the 2d-Euler equations using the flow map inspired by the proof in 
\cite{marchioro2012mathematical}. Indeed we prove the stronger fact that $\omega$ is the unique solution of the 2d Euler equation among data such that only the first 3 terms in the sum given in Equation \eqref{eq:C omega alpha} are required to be finite. In particular this shows that the solution obtained here are obtained as the unique limit of smooth classical solutions of the 2d Euler equation on $[0,T_*]$.

\begin{remark}
    The use of the analytic regularity here does not appear to be a technical artifact but is crucially needed to compensate for a loss of one derivative in closing the estimate to propagate \eqref{eq:C omega alpha}. The loss of one derivative does not appear to be the tightest possible estimate; indeed, the necessary loss appears to be of $\alpha$ derivatives. This would permit the extension of the previous result to data with Gevrey-type regularity for some $\sigma\leq \frac{1}{\alpha}$. In this article we opted to treat the analytic regularity to present the key ideas.
\end{remark}

\subsubsection*{Finite Time Blow Up and Instability}
 We now would like to give a control on the existence time $T_*$ in Theorem \ref{thm:loc fix sing prop} by the control of the motion of the fluid around the origin. The motion of the fluid around the singularity at the origin is described by the $0$-homogeneous Euler equation first introduced in \cite{EJ1}, which was studied extensively in \cite{EMS}. Indeed, if $\omega$ is the solution given by Theorem \ref{thm:loc fix sing prop} (see Section \ref{sec:lim at 0} below) then for all $t\in [0,T_*)$,   we can set \[g(t,\theta)=\displaystyle \lim_{r\to 0}\omega(t,r,\theta).\] Here, $g(t,\theta)$ is odd in and $\frac{2\pi}{m}$ periodic in $\theta$. Furthermore, \[\left\vert \theta \right\vert^{\alpha} g(t,\theta)\in L^\infty\left(0,\frac{\pi}{m}\right)\] and $g(t,\cdot)$ solves
 \begin{equation}\label{SIEuler}\partial_t g + 2G \partial_\theta g =0,\end{equation}
\begin{equation}\label{SIBSLaw}(4+\partial_{\theta\theta})G =g,\end{equation}
with $g(0,\cdot)=g_0(\cdot).$ Theorem 1.1 of \cite{JS} gives local existence for solutions to \eqref{SIEuler}-\eqref{SIBSLaw} with such data.  

The next theorem shows a simple choice of initial data where solutions develop a (further) singularity in finite time. 
\begin{theorem}\label{thm:finite time 0-hom}
    Consider $m\geq3$, and define for $ 0 \leq \theta \leq \frac{\pi}{m}$, $g_0(\theta)=-\theta^{-\alpha}$ which we extend first by oddness to $-\frac{\pi}{m}\leq \theta \leq 0 $ and by $m$-fold symmetry to $\theta \in \mathbb{S}.$ Then the unique solution $g$ defined by Theorem 1.1 of \cite{JS} has a maximal time of existence of the form $(-\infty,T^*)$ with $T^*<+\infty$.
\end{theorem}
\noindent Thus we consider $m\geq3$, and define for $r \geq 0$ and $ 0 \leq \theta \leq \frac{\pi}{m}$, $\omega_0(r,\theta)=-(r+\theta)^{-\alpha}$ which we extend first by oddness to $-\frac{\pi}{m}\leq \theta \leq 0 $ and by $m$-fold symmetry to $\theta \in \mathbb{S}.$ For such initial data, the unique solution $\omega$ defined by Theorem \ref{thm:loc fix sing prop} has a finite maximal time of existence forward in time $T^*<+\infty$. Using the blow up result, we give in Section \ref{sec:non lin inst} an example of a finite time nonlinear instability for weak solutions of the 2d Euler equation in $L^p_{loc}(\mathbb{R}^2)$ for $p>1$. Note that such an instability is not possible in the classes we consider prior to the singularity time. 

\subsection{Some remarks on the proof}
The main type of initial data we consider is $\omega_0(r,\theta)=(r+\theta)^{-\alpha},$ defined first just on $0\leq \theta< \frac{\pi}{m}$ and then extended to be odd and $m$-fold symmetric in $\theta$. It is not difficult to check that the initial velocity field is not $C^{1-\alpha}$ on $\mathbb{R}^2.$ This makes it appear impossible to give any quantitative propagation of regularity or singularity result. One key property, however, is that
\[\sup_{x \in \mathbb{R}^2} \frac{|u_0(x)|}{|x|}<\infty.\] This is not sufficient to propagate regularity but it does indicate that if the solution keeps the same structure as the initial data, then particles cannot be immediately ejected from the point where the vorticity is singular. This means that it is reasonable to believe that the qualitative feature of smoothness away from $r=0$ can be kept, at least for short time. Quantitative loss of smoothness as $(r,\theta)\rightarrow (0,0)$ and truly propagating the singularity is another matter and deeper properties of the equation must be used. 

One important computation is that if we simply consider our initial vorticity as above, we find that the stream function looks smoother than expected due to the regularization provided by the term $\frac{1}{r^2}\partial_{\theta\theta}$ in the Biot-Savart law. In particular, for this type of data, the inverse Laplacian seems to gain \emph{four} derivatives\footnote{Note that a full Cartesian derivative includes $\frac{1}{r}\partial_\theta$, which can be seen to count for two derivatives in $(r,\theta)$ for $r$ small.} in $(r,\theta)$ near $(0,0)$. It can be easily seen that getting the heuristic four-derivative regularization from the inverse Laplacian requires higher regularity on the vorticity than usual (which is what forces us to use analytic-type norms).  A key point is that we work in a space that treats $\partial_r$ derivatives and $\partial_\theta$ derivatives equally. Carefully exploiting the gain of derivatives from the Biot-Savart law and also the structure of the transport term, wherein each term contains exactly one $r$ derivative and one $\theta$ derivative, we are able to close the estimates. 

The finite-time blow-up result is shown by studying carefully solutions to \eqref{SIEuler}-\eqref{SIBSLaw} on $[0,\pi/m]$ with initial data $-\theta^{-\alpha}$ for some $0<\alpha<1.$ By the transport equation and the m-fold symmetry, it is straightforward to show that the solution remains negative and increasing on $[0,\pi/m].$ This, in turn, means that the particles are being pushed to the right, which means that the singularity is being enhanced. Using the monotonicity properties correctly, we can prove an inequality of the form: 
\[-\frac{d}{dt}\int_{0}^{\pi/m} g\geq c\Big(\int_{0}^{\pi/m} g\Big)^2,\] from which a finite-time singularity follows. This singularity, which is based on particle transport, naturally gives rise to an instability by which an infinitesimal bump in \eqref{SIEuler}-\eqref{SIBSLaw} can grow to order 1 in finite time. It is possible to describe the singularity further and show that the singularity is (asymptotically) self-similar; this will be considered in a future paper.

\subsection{Organization of the Paper}
In Section \ref{sec:localexistence}, we give a proof of the local existence of solutions given in Theorem \ref{thm:loc fix sing prop} (assuming some elliptic estimates). In Section \ref{sec:blow-up}, we give a proof of finite-time singularity formation for some solutions in the local existence class given in Theorem \ref{thm:finite time 0-hom}. Finally, in Section \ref{sec:ellipticestimates}, we state and prove the  elliptic estimates we used in the proof of \ref{thm:loc fix sing prop}.

\section*{Acknowledgements}
 T. M. Elgindi acknowledges funding from the NSF DMS-2043024, the Alfred P. Sloan foundation, and a Simons Fellowship.  R. Murray acknowledges partial support from the Simons Foundation MP-TSM. A. R. Said acknowledges funding from the UKRI grant SWAT. A.R. Said would like to thank J. Bedrossian and N. Tzvetkov for insightful comments and discussions on early presentations of this work. The starting idea of this work emanated from previous works with In-Jee Jeong for which the authors are deeply indebted.

\section{Wellposedness Theory}\label{sec:localexistence}
In this section we will give the proof of Theorem \ref{thm:loc fix sing prop}.
\subsection{Propagation of the analytic regularity}
We start from the 2d Euler equation in polar coordinates 
\[
\partial_t \omega+u^r\partial_r \omega+\frac{u^\theta}{r}\partial_\theta\omega=0.
\]
Next we commute with $\partial^k_r$ and $\partial^k_\theta$
\[
\partial_t \partial^{k}_r\omega+\sum^{k}_{i=0}\binom{k}{i}\left(\partial^{k-i}_{r} u^r\partial^{i+1}_r \omega+\partial^{k-i}_{r}\left(\frac{u^\theta}{r}\right)\partial^{i}_r\partial_\theta\omega\right)=0,
\]
thus 
\[
\partial_t \partial^{k}_r\omega+u^r\partial_r \left(\partial^k_r \omega\right)+\frac{u^\theta}{r}\partial_\theta\left(\partial^k_r \omega\right)=-\sum^{k-1}_{i=0}\binom{k}{i}\left(\partial^{k-i}_{r} u^r\partial^{i+1}_r \omega+\partial^{k-i}_{r}\left(\frac{u^\theta}{r}\right)\partial^{i}_r\partial_\theta\omega\right),
\]
and analogously we have 
\[
\partial_t \partial^{k}_\theta\omega+u^r\partial_r \left(\partial^k_\theta \omega\right)+\frac{u^\theta}{r}\partial_\theta\left(\partial^k_\theta \omega\right)=-\sum^{k-1}_{i=0}\binom{k}{i}\left(\partial^{k-i}_{\theta} u^r\partial^{i}_\theta\partial_r \omega+\frac{\partial^{k-i}_{\theta}u^\theta}{r}\partial^{i+1}_\theta\omega\right).
\]
Now we commute with $(r+\theta)^{k+\alpha}$
\begin{align*}
&\partial_t \left((r+\theta)^{k+\alpha}\partial^{k}_r\omega\right)+u^r\partial_r \left((r+\theta)^{k+\alpha}\partial^k_r \omega\right)+\frac{u^\theta}{r}\partial_\theta\left((r+\theta)^{k+\alpha}\partial^k_r \omega\right)\\
&=(k+\alpha)\left(\frac{u^r}{(r+\theta)}+\frac{u^\theta}{r(r+\theta)}\right)(r+\theta)^{k+\alpha}\partial^k_r \omega-(r+\theta)^{k+\alpha}\partial^{k}_{r}\left(\frac{u^\theta}{r}\right)\partial_\theta\omega\\
&-\sum^{k-1}_{i=0}\binom{k}{i}(r+\theta)^{k+\alpha}\partial^{k-i}_{r} u^r\partial^{i+1}_r \omega+\sum^{k-1}_{i=1}\binom{k}{i}(r+\theta)^{k+\alpha}\partial^{k-i}_{r}\left(\frac{u^\theta}{r}\right)\partial^{i}_r\partial_\theta\omega.
\end{align*}
We have singled out the third term in the second line inducing a loss of derivative, which satisfies
\[\|(r+\theta)^{k-1}\partial^{k}_{r}\left(\frac{u^\theta}{r}\right)\|_{L^\infty}\leq \frac{C}{1-\alpha}\left(\left\Vert (r+\theta)^{k+1+\alpha}\partial^{k+1}_r \omega \right\Vert_{L^\infty}+\left\Vert (r+\theta)^{k+\alpha}\partial^{k}_r \omega \right\Vert_{L^\infty}\right),\] by Lemma \ref{lem:est drv utheta/r}.
Analogously we have 
\begin{align*}
&\partial_t \left((r+\theta)^{k+\alpha}\partial^{k}_\theta\omega\right)+u^r\partial_r \left((r+\theta)^{k+\alpha}\partial^k_\theta \omega\right)+\frac{u^\theta}{r}\partial_\theta\left((r+\theta)^{k+\alpha}\partial^k_\theta \omega\right)\\
&=(k+\alpha)\left(\frac{u^r}{(r+\theta)}+\frac{u^\theta}{r(r+\theta)}\right)(r+\theta)^{k+\alpha}\partial^k_\theta \omega\\
&-\sum^{k-1}_{i=0}\binom{k}{i}\left((r+\theta)^{k+\alpha}\partial^{k-i}_{\theta} u^r\partial^{i}_\theta\partial_r  \omega+(r+\theta)^{k+\alpha}\frac{\partial^{k-i}_{\theta}u^\theta}{r}\partial^{i+1}_\theta\omega\right).
\end{align*}
Now we define $f_k(t)=\left\Vert (r+\theta)^{k+\alpha} \partial_r^k\omega (t,\cdot) \right\Vert_{L^\infty} $ and $g_k(t)=\left\Vert (r+\theta)^{k+\alpha}\partial_\theta^k \omega (t,\cdot) \right\Vert_{L^\infty} $. 
From the analytic estimates given in Lemmas \ref{lem:lin u sing}, \ref{lem:est drv ur}, \ref{lem:est drv utheta} and \ref{lem:est drv utheta/r}, along with the transport structure of the left hand side of the previous equalities, we obtain 
\begin{align*}
\partial_t f_k(t)&\leq \frac{C}{(1-\alpha)^2} (k+\alpha)\left(f_0(t)+f_1(t)\right)f_k(t)+\frac{C}{1-\alpha}g_1(t)\left(f_{k+1}(t)+f_{k}(t)\right)\\
&+\frac{C}{1-\alpha}\sum^{k-1}_{i=0}\binom{k}{i} \left(f_{k-i}(t)+(k-i)f_{k-i-1}(t)\right)f_{i+1}(t)\\
&+\frac{C}{1-\alpha}\sum^{k-1}_{i=1}\binom{k}{i} \left(f_{k-i+1}(t)+f_{k-i}(t)\right)\left(f_{i+1}(t)+g_{i+1}(t)\right),
\end{align*}
and 
\begin{align*}
\partial_t g_k(t)&\leq \frac{C}{(1-\alpha)^2} (k+\alpha)\left(f_0(t)+f_1(t)\right)g_k(t)+\frac{C}{(1-\alpha)^2}\sum^{k-1}_{i=0} \binom{k}{i}g_{k-i}(t)\left(g_{i+1}(t)+f_{i+1}(t)\right)\\
&+\frac{C}{1-\alpha}\sum^{k-1}_{i=0} \binom{k}{i} \left(g_{k-i}(t)+f_{k-i}(t)\right)g_{i+1}(t).
\end{align*}

Now we will close an estimate on a quantity of the form $\displaystyle E(t)=\sum^{+\infty}_{k=0} \frac{\lambda(t)^k}{k!}(f_k(t)+g_k(t)),$ and for this we need to introduce $\displaystyle \tilde{E}(t)=\sum^{+\infty}_{k=0} \frac{k\lambda(t)^k}{k!}(f_k(t)+g_k(t))$ which is well defined initially for $\lambda(0)<1$. Now we note that 
\[
\sum^{+\infty}_{k=0}\frac{\lambda(t)}{k!}\left((k+\alpha)\left(f_0(t)+f_1(t)\right)\left(f_k(t)+g_k(t)\right)+g_1(t)\left(f_{k+1}(t)+f_{k}(t)\right)\right)\leq C\left(\tilde{E}(t)+E(t)\right)E(t).
\]
Next we compute 
\begin{align*}
    \sum^{+\infty}_{k=0}\frac{\lambda(t)^k}{k!}\sum^{k-1}_{i=0}\binom{k}{i} f_{k-i}(t)f_{i+1}(t)= \frac{1}{\lambda(t)}\sum^{+\infty}_{k=0}\sum^{k-1}_{i=0} \frac{\lambda(t)^{k-i}f_{k-i}(t)}{(k-i)!}\frac{\lambda(t)^{i+1}f_{i+1}(t)}{(i+1)!}\leq \frac{\tilde{E}(t)E(t)}{\lambda(t)},
\end{align*}
where after re-indexing the sum we used Cauchy's product formula. All of the other terms are treated analogously to finally give 
\[
\partial_t\left(\sum^{+\infty}_{k=0}\frac{\lambda(t)^k}{k!}(f_k(t)+g_k(t))\right)\leq \tilde{E}(t)\frac{\lambda'(t)}{\lambda(t)}+\frac{C}{(1-\alpha)^2}\frac{\tilde{E}(t)E(t)}{\lambda(t)}+\frac{C}{(1-\alpha)^2}\left(\tilde{E}(t)+E(t)\right)E(t),
\]
thus choosing 
\[
\lambda'(t)\leq -\frac{C}{(1-\alpha)^2}E(t),
\]
for a fixed constant $C$ sufficiently large we get $\partial_t E\leq 0$ and $E$ is non increasing giving us the desired result.
\subsection{Uniqueness in $C^\omega_{-\alpha}$}
\noindent In order prove Theorem \ref{thm:loc fix sing prop}, after using the previous apriori estimates it only remains to prove uniqueness of solutions. To this end, by defining the flow map
\[
\frac{d}{dt}\Phi=u\left(t,\Phi(t)\right) \text{ with } \Phi(0,\cdot)=Id,
\]
and writing
\[
\Phi=\left\vert\Phi\right\vert e^{i\mbox{arg} (\Phi)}=\left\vert\Phi\right\vert\cos\left(\mbox{arg} (\Phi)\right)e_x+\left\vert\Phi\right\vert\sin\left(\mbox{arg} (\Phi)\right)e_y=\left\vert\Phi\right\vert\cos\left(\mbox{arg} (\Phi)-\theta\right)e_r+\left\vert\Phi\right\vert\sin\left(\mbox{arg} (\Phi)-\theta\right)e_\theta,
\]
then 
\[
\frac{d}{dt}\left\vert\Phi\right\vert=\frac{u(\Phi)\cdot \Phi}{\left\vert\Phi\right\vert} \text{ and } \frac{d}{dt}\mbox{arg}\ (\Phi)=\frac{u(\Phi)\cdot \Phi^\perp}{\left\vert\Phi\right\vert^2}.
\]
Now we write
\[
u(\Phi)=u^r(\Phi)e_{r}(\mbox{arg}(\Phi))+u^\theta(\Phi)e_{\theta}(\mbox{arg}(\Phi)),
\]
thus 
\[
u(\Phi)\cdot \Phi=u^r(\Phi) \left\vert\Phi\right\vert \text{ and }u(\Phi)\cdot \Phi^\perp=u^\theta(\Phi) \left\vert\Phi\right\vert,
\]
and the system becomes 
\[
\frac{d}{dt}\left\vert\Phi\right\vert=u^r(\Phi) \text{ and } \frac{d}{dt}\mbox{arg}\ (\Phi)=\frac{u^\theta(\Phi)}{\left\vert\Phi\right\vert}.
\]
We now proceed to give the following uniqueness result using the flow map inspired by the proof in \cite{marchioro2012mathematical}. The proof is now an adaptation of the proof given in Lemma 2.17 of \cite{EJ1} but when the velocity has a singularity at the origin and the proof is redone using polar coordinates. Assume there exist two solution triples, $(\omega,u,\Phi)$ and $(\tilde{\omega},\tilde{u},\tilde{\Phi})$, then by Lemmas \ref{lem:lin u sing}-\ref{lem:est drv utheta/r} and the Gr\"onwall lemma for all $\epsilon >0$ there exists $T=T\left(\epsilon,\left\Vert \omega_0\right \Vert_{L^\infty}\right)$ such that for $0 \leq t \leq T$ we have the bounds 
    \begin{equation}\label{eq:shrt time bd flow}
    1-\epsilon \leq \frac{\left\vert \Phi (t,r,\theta)\right\vert}{r},\frac{\left\vert \tilde{\Phi} (t,r,\theta)\right\vert}{r}\leq 1+\epsilon \text{ and } 1-\epsilon \leq \frac{ \mbox{arg}(\Phi)(t,r,\theta)}{\theta},\frac{ \mbox{arg}(\tilde{\Phi})(t,r,\theta)}{\theta}\leq 1+\epsilon.
   \end{equation}
    The goal is to close an estimate for 
    \[  f(t)=\int\limits_{\mathbb{R}_+\times \left(0,\frac{\pi}{m}\right)}\left(\frac{\left\vert \left\vert \Phi\right\vert -\left\vert \tilde{\Phi}\right\vert  \right\vert(t,r,\theta)}{r} +\left(1+(r+\theta)^{-\alpha}\right)\left\vert \mbox{arg}(\Phi)-\mbox{arg}(\tilde{\Phi}) \right\vert(t,r,\theta)\right)\frac{1}{1+r^3}rdrd\theta.
    \]
   For this we compute 
    \[
    \left\vert \Phi\right\vert (t,r,\theta)-\left\vert \tilde{\Phi}\right\vert(t,r,\theta)=\int\limits_0^t u^r(s,\Phi(s,r,\theta))-\tilde{u}^r(s,\tilde{\Phi}(s,r,\theta))ds,\]
    and
    \[\mbox{arg}(\Phi)-\mbox{arg}(\tilde{\Phi})=\int\limits_0^t \frac{u^\theta(s,\Phi(s,r,\theta))}{\left\vert \Phi\right\vert (s,r,\theta)}-\frac{\tilde{u}^\theta(s,\tilde{\Phi}(s,r,\theta))}{\left\vert \tilde{\Phi}\right\vert(s,r,\theta)}ds.\]
    Thus by Corollary \ref{cor:diff pull back} and the Jensen inequality applied to the finite measure $
\frac{1}{1+r^3} r dr d\theta$  we get a bound on $f(t)$ of the form 
\[
f(t)\leq C\left(\sum^{2}_{k=0}\left\Vert(r+\theta)^{k+\alpha}\partial^k_r \omega  \right\Vert_{L^\infty\left(\mathbb{R}_+ \times\left(0,\frac{\pi}{m}\right)\right)}+\left\Vert(r+\theta)^{k+\alpha}\partial^k_\theta \omega  \right\Vert_{L^\infty\left(\mathbb{R}_+ \times\left(0,\frac{\pi}{m}\right)\right)}\right)\int\limits_0^tf(s)ds,
\]
thus by the Gr\"onwall lemma we obtain the desired uniqueness result.
\subsection{Motion of the fluid at the origin}\label{sec:lim at 0}
The goal is to show that $\displaystyle \lim_{r\to 0}\omega(t,r,\theta)=g(t,\theta)$ is a solution of \eqref{SIEuler}-\eqref{SIBSLaw}, which we will show by shiwing that
 $\displaystyle\lim_{r\to 0}\mbox{arg}\ (\Phi)(t,r,\theta)=\mbox{arg}\ (\Phi)(t,0,\theta)$ and that $\mbox{arg}\ (\Phi)(t,0,\theta)$  corresponds to the flow of the corresponding scale invariant solution. We note that it suffices to prove the result on a short interval of time and to then obtain the full result by iteration. We notice that it also suffices to work with $r$ small. Henceforth we will consider $t$ and $r$ sufficiently small such that $\left\vert\Phi\right\vert(t,r,\theta)<1$. Defining $\psi=r^2\underline{\psi}$ as the re-normalised stream function then $\mbox{arg}\ (\Phi)$ solves 
\[
\begin{cases}
\frac{d}{dt}\mbox{arg}\ (\Phi)(t,r,\theta)=2\underline{\psi}\left(t,\left\vert\Phi\right\vert(t,r,\theta),\mbox{arg}\ (\Phi)(t,r,\theta)\right)+\left\vert\Phi\right\vert(t,r,\theta)\partial_r\underline{\psi}\left(t,\left\vert\Phi\right\vert(t,r,\theta),\mbox{arg}\ (\Phi)(t,r,\theta)\right),\\ 
\mbox{arg}\ (\Phi)(0,r,\theta)=\theta.
\end{cases}
\]
By the linear vanishing of the velocity field $\displaystyle\lim_{r\to 0} \left\vert\Phi\right\vert(t,r,\theta)=0$ thus by Lemma \ref{lem:est drv ur} we see that 
\[
\lim_{r\to 0}\underline{\psi}\left(t,\left\vert\Phi\right\vert(t,r,\theta),\mbox{arg}\ (\Phi)(t,r,\theta)\right)=\underline{\psi}\left(t,0,\mbox{arg}\ (\Phi)(t,0,\theta)\right) 
\]
and 
\[
\lim_{r\to 0} \left\vert\Phi\right\vert(t,r,\theta)\partial_r\underline{\psi}\left(t,\left\vert\Phi\right\vert(t,r,\theta),\mbox{arg}\ (\Phi)(t,r,\theta)\right)=0.
\]
Thus $\mbox{arg}\ (\Phi)(t,0,\theta)$ solves 
\[
\begin{cases}
\frac{d}{dt}\mbox{arg}\ (\Phi)(t,0,\theta)=2\underline{\psi}\left(t,0,\mbox{arg}\ (\Phi)(t,0,\theta)\right),\\ 
\mbox{arg}\ (\Phi)(0,r,\theta)=\theta.
\end{cases}
\]
Now define for $r\geq 0$, $f(t,r)=\left\vert \mbox{arg}\ (\Phi)(t,r,\theta)-\mbox{arg}\ (\Phi)(t,0,\theta)\right\vert$.
By Lemmas \ref{lem:lin u sing} and \ref{lem:est drv ur} we get that 
\[
 f(t,r) \leq C\int_0^t  f(t,r) ds+\int_0^tE(r,t)ds,\]
 with $\displaystyle \lim_{r\to 0}\sup_{0\leq t<T^*}E(t,r)=0$. Thus by the Gr\"onwall Lemma for $t<T^*$ we have
 \[
 \lim_{r\to 0}f(t,r)=0 \text{ that is }\lim_{r\to 0}\left\Vert \mbox{arg}\ (\Phi)(t,r,\theta)-\mbox{arg}\ (\Phi)(t,0,\theta)\right\Vert_{L^\infty_\theta}=0.\]
In particular by the transport structure of the equation and the hypothesis on $\omega_0$, $\displaystyle \lim_{r \to 0 }\omega(t,r,\theta)$ exists and is equal to $\omega_0(0,\mbox{arg}\ (\Phi^{-1})(t,0,\theta))=g_0(\mbox{arg}\ (\Phi^{-1})(t,0,\theta))$. 
Now by the singular integral representation of $\underline{\psi}$, combined with the integrability of the kernel in that representation and $\displaystyle \lim_{r \to 0 }\omega(t,r,\theta)=\omega(t,0,\theta)$ we have by the dominated convergence theorem
\[
\underline{\psi}(t,0,\theta)=\frac{m}{4\pi}\int\limits^{+\infty}_0\int\limits^{\frac{\pi}{m}}_{0}\omega(t,0,\phi)
\ln\left(\frac{1-2s^{-m}\cos\left(m(\phi-\theta)\right)+s^{-2m}}{1-2s^{-m} \cos\left(m(\phi+\theta)\right)+s^{-2m}}\right)sd\phi ds.
  \]
    The application of Proposition 3.2 of \cite{EMS} or the proof of Corollary 3.14 of \cite{EJ1} to scale-invariant vorticities which are smooth in $\theta$ then gives that 
    \[
    \underline{\psi}(t,0,\theta)=\int\limits^{\frac{\pi}{m}}_{-\frac{\pi}{m}}\frac{3\pi g_0(m(\mbox{arg}\ (\Phi^{-1})(t,r,\phi))}{2m(m^2-4)}\left|\sin\left(\frac{m}{2}\left(m(\theta-\phi\right)\right)\right|d\phi
    \]
    thus 
\[
\begin{cases}
    \frac{d}{dt}\mbox{arg}\ (\Phi)(t,0,\theta)=\frac{3\pi }{2m(m^2-4)}\int^{\frac{\pi}{m}}_{-\frac{\pi}{m}}g_0(m(\mbox{arg}\ (\Phi^{-1})(t,r,\phi))\left|\sin\left(\frac{m}{2}\left(m(\mbox{arg}\ (\Phi)(t,0,\theta)-\phi\right)\right)\right|d\phi,\\
    \frac{d}{dt}\mbox{arg}\ (\Phi)(t,0,\theta)=\theta,
\end{cases}
\]
which is exactly the flow of \eqref{SIEuler}-\eqref{SIBSLaw}. 

\section{Finite time blow up for $C^\omega_{-\alpha}$ vorticity and some consequences} \label{sec:blow-up}
This section is devoted to proving the blow-up of solutions to \eqref{SIEuler}-\eqref{SIBSLaw} and giving a few applications.
\subsection{Proof of Theorem \ref{thm:finite time 0-hom}}
 In order to work with positive quantities, we will work with $g_0(\theta)=\theta^{-\alpha}$ and reverse the sign of the Biot-Savart law. Note that this makes $g$ and $G$ of the same sign. More precisely we are working with initial value problem
\[
\begin{cases}
    \partial_t g+2G\partial_\theta g=0,\\
    4G+\partial_{\theta \theta}G=-g,
\end{cases}\text{ with } g_0=\theta^{-\alpha} \text{ for }0 \leq \theta \leq \frac{\pi}{m}.
\]
Using odd symmetry we write
\[
G\left(t,\theta\right)=\frac{3m}{2\left(m^2-4\right)}\left(\cos\left(\frac{m}{2}\theta\right)\int_0^\theta\sin\left(\frac{m}{2}\phi\right)g(t,\phi)d\phi+\sin\left(\frac{m}{2}\theta\right)\int_\theta^{\frac{\pi}{m}}\cos\left(\frac{m}{2}\phi\right)g(t,\phi)d\phi\right).
\]
Thus by symmetry and positivity of $g$ we have the following properties
\begin{itemize}
    \item $G$ is positive on $(0,\frac{\pi}{m})$ and odd.
    \item $g(t,\theta)$ is increasing in $t$ for all $\theta$ and is decreasing in $\theta$ for all $t$.
    \end{itemize} 
    \noindent We compute the time evolution of $\int_0^{\frac{\pi}{m}}g(t,\theta)d\theta$ thus 
\[
\partial_t\left(\int_0^{\frac{\pi}{m}}g(t,\theta)d\theta\right)+2\int_0^{\frac{\pi}{m}}G(t,\theta)\partial_\theta g(t,\theta)d\theta=0,
\]
thus
\[
\partial_t\left(\int_0^{\frac{\pi}{m}}g(t,\theta)d\theta\right)+2\int_0^{\frac{\pi}{m}}\partial_\theta G(t,\theta) \left(4G(t,\theta)+\partial_{\theta \theta}G(t,\theta)\right)d\theta=0,
\]
which gives 
\[
\partial_t\left(\int_0^{\frac{\pi}{m}}g(t,\theta)d\theta\right)=\left(\partial_\theta G\left(t,0\right)-\partial_\theta G\left(t,\frac{\pi}{m}\right)\right)\left(\partial_\theta G\left(t,0\right)+\partial_\theta G\left(t,\frac{\pi}{m}\right)\right).
\]
We also compute 
\[
\partial_\theta G\left(t,0\right)-\partial_\theta G\left(t,\frac{\pi}{m}\right)=\int^{\frac{\pi}{m}}_0g(t,\theta)d\theta +4\int^{\frac{\pi}{m}}_0G(t,\theta)d\theta \geq  \int^{\frac{\pi}{m}}_0g(t,\theta)d\theta,
\]
thus 
\[
\partial_t\left(\int_0^{\frac{\pi}{m}}g(t,\theta)d\theta\right)\geq \left(\int_0^{\frac{\pi}{m}}g(t,\theta)d\theta\right)\left(\partial_\theta G\left(t,0\right)+\partial_\theta G\left(t,\frac{\pi}{m}\right)\right).
\]
The goal now is to show a bound of the form
\[
\partial_\theta G\left(t,0\right)+\partial_\theta G\left(t,\frac{\pi}{m}\right)\geq C \int_0^{\frac{\pi}{m}}g(t,\theta)d\theta,
\]
which combined with $\int_0^{\frac{\pi}{m}}g_0(\theta)d\theta>0$ gives the desired finite time blow up. For this we define the characteristic flow 
\[
\frac{d}{dt}\chi(t,\theta)=2G\left(t,\chi(t,\theta)\right), \ \chi(0,\theta)=0,
\] 
for which we prove the following properties.
\begin{lemma}\label{lem:increasing ratios}
 Consider $0<\theta_1<\theta_2<\frac{\pi}{m}$ then 
\[
\frac{\chi(t,\theta_1)}{\chi(t,\theta_2)} \text{ is strictly increasing in time}.
\]
\end{lemma}
\begin{proof}
    We compute 
    \begin{align*}
    \frac{d}{dt}\frac{\chi(t,\theta_1)}{\chi(t,\theta_2)}&=2\frac{G(t,\Phi(t,\theta_1))\chi(t,\theta_2)-\chi(t,\theta_1)G(t,\chi(t,\theta_2))}{\chi(t,\theta_2)^2}\\
    &=2\frac{\chi(t,\theta_1)}{\chi(t,\theta_2)}\left(\frac{G(t,\chi(t,\theta_1))}{\chi(t,\theta_1)}-\frac{G(t,\chi(t,\theta_2))}{\chi(t,\theta_2)}\right),
    \end{align*}
    and observe $\frac{G(t,\chi(t,\theta_1))}{\chi(t,\theta_1)}-\frac{G(t,\chi(t,\theta_2))}{\chi(t,\theta_2)}>0$ by strict concavity of $G$ which follows from the small domain maximum principle.
\end{proof}
\begin{corollary}\label{lem:true var}
    Consider $0<\eta_1<\eta_2<\frac{\pi}{m}$ then 
    \[
    \frac{g(t,\eta_2)}{g(t,\eta_1)}\leq \left(\frac{\eta_2}{\eta_1}\right)^{-\alpha}.
    \]
\end{corollary}
\begin{proof}
     Define $\chi(t,\theta_1)=\eta_1$ and $\chi(t,\theta_2)=\eta_2$. Then
    \[
    \frac{g(t,\eta_2)}{g(t,\eta_1)}=\frac{g_0(\theta_2)}{g_0(\theta_1)}=\left(\frac{\theta_2}{\theta_1}\right)^{-\alpha},
    \]
    Now by Lemma \eqref{lem:increasing ratios} $\left(\frac{\theta_2}{\theta_1}\right)^{-\alpha}\leq \left(\frac{\chi(t,\theta_2)}{\chi(t,\theta_1)}\right)^{-\alpha}$ thus 
    \[
     \frac{g(t,\eta_2)}{g(t,\eta_1)}\leq \left(\frac{\eta_2}{\eta_1}\right)^{-\alpha}<1.
    \]
\end{proof}
\noindent We now compute 
  \[
\partial_\theta G\left(t,\theta\right)=\frac{3m^2}{4\left(m^2-4\right)}\left(-\sin\left(\frac{m}{2}\theta\right)\int_0^\theta\sin\left(\frac{m}{2}\phi\right)g(t,\phi)d\phi+\cos\left(\frac{m}{2}\theta\right)\int_\theta^{\frac{\pi}{m}}\cos\left(\frac{m}{2}\phi\right)g(t,\phi)d\phi\right),
\]
thus 
\[
\partial_\theta G\left(t,0\right)-\partial_\theta G\left(t,\frac{\pi}{m}\right)=\frac{3m^2}{4\left(m^2-4\right)}\int_0^{\frac{\pi}{m}}\left(\cos\left(\frac{m}{2}\phi\right)-\sin\left(\frac{m}{2}\phi\right)\right)g(t,\phi)d\phi.
\]
Analysing more closely the integral on the right hand side we get
\begin{align*}
    &\int_0^{\frac{\pi}{m}}\left(\cos\left(\frac{m}{2}\phi\right)-\sin\left(\frac{m}{2}\phi\right)\right)g(t,\phi)d\phi\\
    &=\int_0^{\frac{\pi}{2m}}\left(\cos\left(\frac{m}{2}\phi\right)-\sin\left(\frac{m}{2}\phi\right)\right)\left(g(t,\phi)-g\left(t,\frac{\pi}{m}-\phi\right)\right)d\phi\\
    &=\int_0^{\frac{\pi}{2m}}\left(\cos\left(\frac{m}{2}\phi\right)-\sin\left(\frac{m}{2}\phi\right)\right)g(t,\phi)\left(1-\frac{g\left(t,\frac{\pi}{m}-\phi\right)}{g(t,\phi)}\right)d\phi\\
    &\geq \int_0^{\frac{\pi}{2m}}\left(\cos\left(\frac{m}{2}\phi\right)-\sin\left(\frac{m}{2}\phi\right)\right)g(t,\phi)\left(1-\left(\frac{\left(\frac{\pi}{m}-\phi\right)}{\phi}\right)^{-\alpha}\right)d\phi\\
    &\geq C \int_0^{\frac{\pi}{2m}}\left(\cos\left(\frac{m}{2}\phi\right)-\sin\left(\frac{m}{2}\phi\right)\right)g(t,\phi)d\phi\\
    &\geq C' \int_0^{\frac{\pi}{4m}}g(t,\phi)d\phi\\
    &\geq C'' \int_0^{\frac{\pi}{m}}g(t,\phi)d\phi
\end{align*}
where we used that $g$ is decreasing in $\theta$ in the last inequality, which completes the proof.

\subsection{Nonlinear instability in $L^p_{loc}\left(\mathbb{R}^2\right)$ of the 2d Euler equation}\label{sec:non lin inst}
We will now give a family of weak solutions to the 2d Euler equations in vorticity form for vorticity in $L^p_{loc}\left(\mathbb{R}^2\right)$ exhibiting a nonlinear instability in finite time, thus proving the possibility of bifurcation for weak solutions of the 2d Euler equations in $L^p_{loc}\left(\mathbb{R}^2\right)$. This starts from the observation that by Proposition 1.7 of \cite{JS} a solution $g$ of \eqref{SIEuler}-\eqref{SIBSLaw} defines a weak solution of the 2d Euler equation in vorticity form in $C^0\left([0,T],L^p_{loc}\left(\mathbb{R}^2\right)\right)$. Next let us recall that a natural topology on $L^p_{loc}\left(\mathbb{R}^2\right)$ is given by the following metric
\[
d_{L^p_{loc}\left(\mathbb{R}^2\right)}(f,g)=\sum^{+\infty}_{n=0}2^{-n}\left(\frac{1}{\left|B(0,2^n)\right|}\int_{B(0,2^n)}\left\vert f(x)-g(x)\right\vert^p dx\right)^\frac{1}{p}.
\]
We note that for scale invariant functions, i.e functions only depending on $\theta$, we have that 
\[
d_{L^p_{loc}\left(\mathbb{R}^2\right)}(f(\theta),g(\theta))=  \left\Vert f-g \right\Vert_{L^p\left(\mathbb{S}^1\right)}.
\]
By Theorems 1.1 and Proposition 1.7 of \cite{JS}, the $m$-periodic odd extension of $g_0(\theta)=-\theta^{-\alpha}$ defines a weak solution  $g(t,\theta)\in C\left((-\infty,T^*),L^p_{loc}\left(\mathbb{R}^2\right)\right)$ which blows up at $T^*$. We are now going to exploit this blow up mechanism and the time reversibility of the Euler equation to construct our instability; in particular consider the solution $g(t,\theta)$ given in Theorem \ref{thm:finite time 0-hom} and an $0<\epsilon \ll 1$. Now we define $g^{\epsilon}$ starting with data prescribed at $T^*-\epsilon$ by \[
g^{\epsilon}(T^*-\epsilon,\theta)=\begin{cases}
    g(T^*-\epsilon,\theta), \text{ for } \theta \leq \frac{\pi}{m}-\epsilon,\\
    0  , \text{ for } \theta \geq \frac{\pi}{m}-\epsilon,
\end{cases}
\]
then there exists $T^{*,\epsilon}\geq T^*$ such that $g^{\epsilon}(t,\theta)\in C\left((-\infty,T^{*,\epsilon}),L^p_{loc}\left(\mathbb{R}^2\right)\right)$ defines a weak solution to the Euler equation. Moreover we have the following separation estimate.
\begin{proposition}\label{prop:inf dim bif lp loc}
     There exists $c>0$ such that 
    \[
    \lim_{\epsilon\to 0}\left\Vert g^{\epsilon}(T^*-\epsilon,\cdot)-g(T^*-\epsilon,\cdot)\right\Vert_{L^p\left(\mathbb{S}\right)}=\lim_{\epsilon\to 0}d_{L^p_{loc}\left(\mathbb{R}^2\right)}(g^{\epsilon}(T^*-\epsilon,\theta),g(T^*-\epsilon,\theta))=0 \text{ for }\frac{1}{\alpha}>p,
    \]
    and 
    \[
    \left\Vert g^{\epsilon}(0,\cdot)-g(0,\cdot)\right\Vert_{L^p\left(\mathbb{S}\right)}=d_{L^p_{loc}\left(\mathbb{R}^2\right)}(g^{\epsilon}(0,\theta),g(0,\theta))>c \text{ for }\frac{1}{\alpha}>p.
    \]
\end{proposition}
\begin{proof}
The first limit is immediate from the definition of $g^{\epsilon}$. The second limit follows using the analysis in \cite{EMS} along with the fact that $g_0$ blows up at time $T^*$. As $g^{\epsilon}$ are m-fold symmetric and odd, and $g$ is negative on $(0,\pi/m)$, we have that the velocity field $G$ is strictly positive on $(0,\pi/m)$. As $g$ satisfies a transport equation with the velocity $G$, we have that as time goes to $-\infty$ that $g$ approaches the constant $-(\pi/m)^{-\alpha}$, whereas $g^{\epsilon}$ approaches the constant zero, independent of $\epsilon$. The only task then is to show that this occurs uniformly in $\epsilon$ at time $t = 0$: this occurs precisely because $g$ blows up at time $T^*$ with $\int_{\mathbb{S}}g$ having a Riccati-type blow up as shown in the proof of Theorem \ref{thm:finite time 0-hom} and more precisely due to the estimate \ref{eq: flw sep est}.
\end{proof}
\noindent The previous proposition thus proves instability of weak solutions of the Euler equation in $L^p_{loc}\left(\mathbb{R}^2\right)$ for $p<\frac{1}{\alpha}$ and $0<\alpha<1$. Let us note that on the interval $\left(\frac{\pi}{m}-\epsilon,\frac{\pi}{m}\right)$ we could have  chosen an arbitrary continuous function  instead of the trivial function thus showing that the previous instability is infinite dimensional.\\

\subsubsection{Application: a potential non uniqueness scenario for the unforced 2d Euler equation in $L^p$}

We will now give a localisation of the previous instability using Section \ref{sec:lim at 0} which permits the construction of the following instability of the motion at the origin of the solution map for Euler equations for initial vorticity in $L^p(\mathbb{R}^2)$. Indeed we note that if one is allowed a low regularity forcing term then non uniqueness has been shown in a recent breakthrough by Vishik \cite{Vishik1,Vishik2} for the 2d Euler equations, see also the excellent notes \cite{albritton2021instability}.

\noindent We first smooth out the instability at the level of the $0$-homogeneous Euler equations \eqref{SIEuler}-\eqref{SIBSLaw}. For this we consider again the solution $g(t,\theta)$ given in Theorem \ref{thm:finite time 0-hom} and an $0<\epsilon \ll 1$. Note that $g$ is always strictly increasing on $(0,\frac{\pi}{m})$ and thus $g^{-1}$ is well defined. Then we define first
\[
g^{\epsilon,1}(T^*-\epsilon,\theta)=\begin{cases}
    -\epsilon^{-\alpha}, \text{ for } 0\leq \theta \leq g^{-1}(T^*-\epsilon,-\epsilon^{-\alpha}),\\
    g(T^*-\epsilon,\theta), \text{ for } \theta \geq g^{-1}(T^*-\epsilon,-\epsilon^{-\alpha})+\epsilon,
\end{cases}
\]
and $g^{\epsilon,1}(T^*-\epsilon,\cdot)$ is extended to $[g^{-1}(T^*-\epsilon,-\epsilon^{-\alpha}),g^{-1}(T^*-\epsilon,-\epsilon^{-\alpha})+\epsilon]$ in an increasing smooth fashion. Now from $g^{\epsilon,1}(T^*-\epsilon,\cdot)$ we define
\[
g^{\epsilon,2}(T^*-\epsilon,\theta)=\begin{cases}
    g^{\epsilon,1}(T^*-\epsilon,\theta), \text{ for } \theta \leq \frac{\pi}{m}-2\epsilon,\\
    0  , \text{ for } \theta \geq \frac{\pi}{m}-\epsilon,
\end{cases}
\]
and $g^{\epsilon,2}(T^*-\epsilon,\cdot)$ is extended to $[\frac{\pi}{m}-2\epsilon,\frac{\pi}{m}-\epsilon]$ in an increasing smooth fashion.
Thus $g^{\epsilon,1}(T^*-\epsilon,\cdot)$ and $g^{\epsilon,2}(T^*-\epsilon,\cdot)$ define two bounded, $m$-symmetric smooth functions on $\mathbb{S}$ that we extended to globally smooth functions $g^{\epsilon,1}(t,\cdot)$ and $g^{\epsilon,2}(t,\cdot)$ for $t\in \mathbb{R}$ as the unique bounded global solutions of \eqref{SIEuler}-\eqref{SIBSLaw} with data $g^{\epsilon,1}(T^*-\epsilon,\cdot)$ and $g^{\epsilon,2}(T^*-\epsilon,\cdot)$ prescribed at $T^*-\epsilon$. To extend the result to the 2d Euler equations we consider $\omega^{\epsilon,1}(T^*-\epsilon,r,\theta)$ as an extension of $g^{\epsilon,1}(T^*-\epsilon,\theta)$ that is decreasing along (radial) rays and is analytic. We choose  $\omega^{\epsilon,2}(T^*-\epsilon,r,\theta)$ that is analytic outside of the origin and such that $\left\Vert \omega^{\epsilon,1}(T^*-\epsilon,\cdot)-\omega^{\epsilon,2}(T^*-\epsilon,\cdot)\right\Vert_{L^p\left(\mathbb{R}^2\right)}\leq \epsilon$ and $\displaystyle \lim_{r\to 0}\omega^{\epsilon,2}(T^*-\epsilon,r,\theta)=g^{\epsilon,2}(T^*-\epsilon,\theta) $.
We extend them as global solutions to \eqref{eq: 2dEuler1}-\eqref{eq: 2dEuler2} with data prescribed at $T^*-\epsilon$ by $\omega^{\epsilon,1}(T^*-\epsilon,\cdot)$ and $\omega^{\epsilon,2}(T^*-\epsilon,\cdot)$ respectively.  Then we have the following:
\begin{proposition}\label{prop:instab 0-hom}
    There exists $c>0$ such that 
    \[
    \lim_{\epsilon\to 0}\left\Vert g^{\epsilon,1}(T^*-\epsilon,\cdot)-g^{\epsilon,2}(T^*-\epsilon,\cdot)\right\Vert_{L^p\left(\mathbb{S}\right)}=0 \text{ for }\frac{1}{\alpha}>p,
    \]
    and 
    \[
    \lim_{\epsilon\to 0}\left\Vert g^{\epsilon,1}(0,\cdot)-g^{\epsilon,2}(0,\cdot)\right\Vert_{L^p\left(\mathbb{S}\right)}>c \text{ for }\frac{1}{\alpha}>p.
    \]
    Thus by Section \ref{sec:lim at 0} there exists $c>0$ such that 
    \[
    \lim_{\epsilon\to 0}\left\Vert \omega^{\epsilon,1}(T^*-\epsilon,\cdot)-\omega^{\epsilon,2}(T^*-\epsilon,\cdot)\right\Vert_{L^p\left(\mathbb{R}^2\right)}=0 \text{ for }\frac{3}{\alpha}>p,
    \]
    and 
    \[
    \liminf_{\epsilon\to 0}\left\Vert \omega^{\epsilon,1}(0,0,\cdot)-\omega^{\epsilon,2}(0,0,\cdot)\right\Vert_{L^p\left(\mathbb{S}\right)}>c \text{ for }\frac{1}{\alpha}>p.
    \]
\end{proposition}

\begin{remark}
    It is of great interest to now if the previous instability of the motion of the fluid at the origin can be propagated to the bulk of the fluid.
\end{remark}

\subsection{Separation of trajectories estimate}
In order to show separation of trajectories in Propositions \ref{prop:inf dim bif lp loc} and \ref{prop:instab 0-hom} we need to show the following estimate on the solution constructed in the proof of Theorem \ref{thm:finite time 0-hom}, see Section \ref{sec:blow-up}.  There exists $0<c<\frac{\pi}{m}$ such that for all $0<\epsilon\ll 1$:
\begin{equation}\label{eq: flw sep est}
\chi(\epsilon,T^*-\epsilon)\geq c\text{ and } \chi^{-1}\left(\frac{\pi}{m}-\epsilon,T^*-\epsilon\right)\leq \frac{\pi}{m}-c.
\end{equation}
Those estimates follow from the Riccati blow up of $\int_0^{\frac{\pi}{m}}g(t,\theta)d\theta$ above and the following observations from the representation formula of $G$:
\[
G(T^*-\epsilon,\epsilon)=\frac{m}{2}\epsilon\int_0^{\frac{\pi}{m}}g\left(T^*-\epsilon,\theta\right)d\theta+O(\epsilon),
\]
which implies the first uniform estimate and 
\[
G(T^*-\epsilon,\epsilon)=\frac{m}{2}\epsilon \chi(\epsilon,T^*-\epsilon)\int^{\frac{\pi}{m}}_0g(t,\theta)d\theta+O(\epsilon),
\]
which gives the second uniform estimate.

\section{The Biot-Savart law in polar coordinates}\label{sec:ellipticestimates}
\noindent By the oddness and m-fold symmetry hypothesis we write
\[
\omega(r,\theta)=\sum^{+\infty}_{k=1}\omega_{k}(r)\sin(mk\theta) \text{ and }\psi(r,\theta)=\sum^{+\infty}_{k=1}\psi_{k}(r)\sin(mk\theta).
\]
Solving the Poisson equation we get
\begin{equation}\label{eq:solv lap 1}
        \psi_{k}(r)=-\frac{r^{mk}}{2mk}\int\limits\limits_r^{+\infty}\omega_k(s)s^{1-mk}ds-\frac{r^{-mk}}{2mk}\int\limits\limits_0^rs^{mk+1}\omega_k(s)ds.
    \end{equation}
 As $L^\infty$-based estimates are not immediate when taking Fourier series we use the more standard and  natural approach of working with the corresponding singular kernels. From \eqref{eq:solv lap 1} and by applying Lemma \ref{lem:inf sum comp} we get $
\psi=\boldsymbol{\psi}_1+\boldsymbol{\psi}_2,$ with 
\begin{equation*}
   \boldsymbol{\psi}_1(r,\theta)=\frac{m}{8\pi }\int\limits\limits_r^{+\infty}\int\limits\limits_{-\frac{\pi}{m}}^{\frac{\pi}{m}}\omega(s,\phi)\ln\left(\frac{1-2\left(\frac{s}{r}\right)^{-m}\cos\left(m(\phi-\theta)\right)+\left(\frac{s}{r}\right)^{-2m}}{1-2\left(\frac{s}{r}\right)^{-m} \cos\left(m(\phi+\theta)\right)+\left(\frac{s}{r}\right)^{-2m}}\right)sd\phi ds,
\end{equation*}
\begin{equation*}
    \boldsymbol{\psi}_2(r,\theta)=\frac{m}{8\pi }\int\limits\limits_0^{r}\int\limits\limits_{-\frac{\pi}{m}}^{\frac{\pi}{m}}\omega(s,\phi)\ln\left(\frac{1-2\left(\frac{s}{r}\right)^{m}\cos\left(m(\phi-\theta)\right)+\left(\frac{s}{r}\right)^{2m}}{1-2\left(\frac{s}{r}\right)^{m} \cos\left(m(\phi+\theta)\right)+\left(\frac{s}{r}\right)^{2m}}\right)sd\phi ds,
\end{equation*}
thus rearranging the integrals we get 
\begin{equation}\label{eq:psi BS kernel polar}
  \psi(r,\theta)=  \frac{m}{4\pi }\int\limits\limits_0^{+\infty}\int\limits\limits_{0}^{\frac{\pi}{m}}\omega(s,\phi)\ln\left(\frac{1-2\left(\frac{s}{r}\right)^{-m}\cos\left(m(\phi-\theta)\right)+\left(\frac{s}{r}\right)^{-2m}}{1-2\left(\frac{s}{r}\right)^{-m} \cos\left(m(\phi+\theta)\right)+\left(\frac{s}{r}\right)^{-2m}}\right)sd\phi ds.
\end{equation}
Now taking derivatives in \eqref{eq:psi BS kernel polar} we get 
\begin{equation}\label{eq:u r kernel polar}
    u^r(r,\theta)=-\frac{1}{r}\partial_\theta \psi(r,\theta)=\frac{m^2}{2\pi }\int\limits\limits_0^{+\infty}\int\limits\limits_{0}^{\frac{\pi}{m}}\frac{\sin\left(m(\phi-\theta)\right)\omega(s,\phi)\left(\frac{s}{r}\right)^{-m}}{1-2\left(\frac{s}{r}\right)^{-m} \cos\left(m(\phi-\theta)\right)+\left(\frac{s}{r}\right)^{-2m}}\frac{1}{r}sd\phi ds.
\end{equation}
\[u^\theta=U^\theta_1+U^\theta_2 \text{ with},\]
\begin{equation}\label{eq:U BS kernel polar 1}
   U^\theta_1=\frac{m^2}{2\pi }\int\limits\limits_r^{+\infty}\int\limits\limits_{-\frac{\pi}{m}}^{\frac{\pi}{m}}\frac{\omega(s,\phi)\left(\left(\frac{s}{r}\right)^{-m}-\cos\left(m(\phi-\theta)\right)\right)}{1-2\left(\frac{s}{r}\right)^{-m} \cos\left(m(\phi-\theta)\right)+\left(\frac{s}{r}\right)^{-2m}}\left(\frac{s}{r}\right)^{-m}\frac{1}{r}sd\phi ds,
\end{equation}
\begin{equation}\label{eq:U BS kernel polar 2}
    U^\theta_2=\frac{m^2}{2\pi }\int\limits\limits_0^{r}\int\limits\limits_{-\frac{\pi}{m}}^{\frac{\pi}{m}}\frac{\omega(s,\phi)\left(\cos\left(m(\phi-\theta)\right)-\left(\frac{s}{r}\right)^{m}\right)}{1-2\left(\frac{s}{r}\right)^{m} \cos\left(m(\phi-\theta)\right)+\left(\frac{s}{r}\right)^{2m}}\left(\frac{s}{r}\right)^{m}\frac{1}{r}sd\phi ds.
\end{equation}

\subsection{First derivatives}
A good guide (and check!) through the rather technical computations below is the following observation: within $m-$fold and odd-odd symmetry one can can compute explicitly near $(r,\theta)=(0,0)$ that $$\Delta^{-1}\left((r+\theta)^{-\alpha}\right)=C r^2\theta+O(r^2(r+\theta)^{2-\alpha}).$$
\begin{lemma}[Linear growth of the velocity]\label{lem:lin u sing}
    For $0\leq \alpha<1$ there exists $C\in \mathbb{R}^*_+$ such that for all $\omega \in \mathscr{S}(\mathbb{R}^2)$
    \[
    \left\vert u^r(r,\theta)\right\vert\leq \frac{C}{1-\alpha} r \left\Vert (r+\theta)^{\alpha} \omega \right\Vert_{L^\infty}, \text{ and }  \left\vert u^\theta(r,\theta) \right\vert \leq \frac{C}{(1-\alpha)^2} r \left\Vert (r+\theta)^{\alpha} \omega \right\Vert_{L^\infty}
    \]
\end{lemma}
\begin{proof}
 We write
\begin{equation*}
    u^r=r\frac{m^2}{2\pi }\int\limits\limits_0^{+\infty}\int\limits\limits_{0}^{\frac{\pi}{m}}\frac{\sin\left(m(\phi-\theta)\right)\omega(r\rho,\phi)\rho^{-m}}{1-2\rho^{-m} \cos\left(m(\phi-\theta)\right)+\rho^{-2m}}\rho d\phi d\rho,
\end{equation*}
\begin{equation*}
    U^\theta_1=r\frac{m^2}{2\pi }\int\limits\limits_1^{+\infty}\int\limits\limits_{-\frac{\pi}{m}}^{\frac{\pi}{m}}\frac{\omega(r\rho,\phi)\left(\rho^{-m}-\cos\left(m(\phi-\theta)\right)\right)}{1-2\rho^{-m} \cos\left(m(\phi-\theta)\right)+\rho^{-2m}}\rho^{-m}\rho d\phi d\rho.
\end{equation*}
It now suffices to see that 
    \[
\begin{cases}
    (1)=\int\limits\limits_0^{+\infty}\int\limits\limits_{0}^{\frac{\pi}{m}}\frac{(r\rho+\left|\phi\right|)^{-\alpha}\left\vert\sin\left(m(\phi-\theta)\right)\right\vert \rho^{-m}}{1-2\rho^{-m} \cos\left(m(\phi-\theta)\right)+\rho^{-2m}}\rho d\phi d\rho\leq\frac{C_\alpha}{1-\alpha},\\
    (2)=\int\limits\limits_1^{+\infty}\int\limits\limits_{-\frac{\pi}{m}}^{\frac{\pi}{m}}\frac{(r\rho+\left|\phi\right|)^{-\alpha}\left(\rho^{-m}-\cos\left(m(\phi-\theta)\right)\right)}{1-2\rho^{-m} \cos\left(m(\phi-\theta)\right)+\rho^{-2m}}\rho^{-m}\rho d\phi d\rho\leq\frac{C_\alpha}{(1-\alpha)^2}.
\end{cases}
\]
Indeed for $(1)$ as $m\geq 3$, $1-m\leq -2$ which gives integrability at $+\infty$. The singular part of the integral occurs for $\rho=1$ and $\phi=\theta$ and is most severe for $r=\theta=0$. Thus, after standard computations where we see the denominator as $(1-\rho^{-m})^2+2\rho^{-m}(1-\cos(m(\phi-\theta)))$, it suffices to consider
\[\int\limits\limits_0^{1}\int\limits\limits_{0}^{1}\frac{\phi^{1-\alpha} }{y^2+\phi^2}dyd\phi=\int\limits\limits_0^{1}\phi^{-\alpha}\arctan\left(\frac{1}{\phi}\right)d\phi\leq \frac{\pi}{2}\frac{1}{1-\alpha}.\]
Now for $(2)$, again $m\geq 3$ gives integrability at infinity and the singularity in the integral occurs for $\rho=1$ and $\phi=\theta$ and is most severe for $r=\theta=0$. Thus, after standard computations, it suffices to consider
\[\int\limits\limits_0^{1}\int\limits\limits_{0}^{1}\frac{\phi^{-\alpha}\left(y+\phi^2\right) }{y^2+\phi^2}dyd\phi\leq C+\int\limits\limits_0^{1}\int\limits\limits_{0}^{1}\frac{\phi^{-\alpha}y }{y^2+\phi^2}dyd\phi,\]
thus 
\[
\int\limits\limits_0^{1}\int\limits\limits_{0}^{1}\frac{\phi^{-\alpha}y }{y^2+\phi^2}dyd\phi\leq \int\limits_0^{1}\phi^{-\alpha}\ln\left(\frac{1}{\phi}\right)d\phi\leq \frac{1}{(1-\alpha)^2}.
\]
Finally $U^\theta_2$ is treated analogously. 
\end{proof}
Now we turn to higher derivatives.
\begin{lemma}\label{lem:est drv ur}
    For $0\leq \alpha<1$  and $k\in \mathbb{N}$, $k\geq 1$ there exists $C\in \mathbb{R}^*_+$  such that for all $\omega \in \mathscr{S}(\mathbb{R}^2)$ and $k\in \mathbb{N}$, $k\geq 1$
    \[
    \left\vert \frac{(r+\theta)^{k-1}}{1+(r+\theta)^{k-1}}\partial^{k}_r u^r(r,\theta)\right\vert \\ \leq \frac{C}{1-\alpha} \left(\left\Vert (r+\theta)^{k+\alpha}\partial^{k}_r \omega \right\Vert_{L^\infty}+k\left\Vert (r+\theta)^{k-1+\alpha}\partial^{k-1}_r \omega \right\Vert_{L^\infty}\right),\]   
    and 
    \[ \left\vert \frac{(r+\theta)^{k-1}}{1+(r+\theta)^{k-1}} \frac{\partial_\theta^k u^r}{r}(r,\theta)\right\vert\leq \frac{C}{1-\alpha} \left(1+(r+\theta)^{-\alpha}\right)  \left\Vert (r+\theta)^{k+\alpha}\partial^k_\theta \omega \right\Vert_{L^\infty}.
    \]
\end{lemma}
 \begin{proof}
        Starting from \begin{equation*}
    u^r=r\frac{m^2}{2\pi }\int\limits\limits_0^{+\infty}\int\limits\limits_{0}^{\frac{\pi}{m}}\frac{\sin\left(m(\phi-\theta)\right)\omega(r\rho,\phi)\rho^{-m}}{1-2\rho^{-m} \cos\left(m(\phi-\theta)\right)+\rho^{-2m}}\rho d\phi d\rho,
\end{equation*}
and noting that $\partial_x^k(xf(x))=xf^{(k)}(x)+kf^{(k-1)}(x)$ we see that for $\frac{(r+\theta)^{k-1}}{1+(r+\theta)^{k-1}}\partial_r^{k} u^r$ we need to estimate 
\[
r\frac{(r+\theta)^{k-1}}{1+(r+\theta)^{k-1}}\int\limits\limits_0^{+\infty}\int\limits\limits_{0}^{\frac{\pi}{m}}\frac{\sin\left(m(\phi-\theta)\right) \rho^k (r\rho+\left|\phi\right|)^{-k-\alpha}\rho^{-m}}{1-2\rho^{-m} \cos\left(m(\phi-\theta)\right)+\rho^{-2m}}\rho d\phi d\rho.
\]
Again $m\geq 3$ gives integrability at infinity and the singularity in the integral occurs for $\rho=1$ and $\phi=\theta$ and is most severe for $r=\theta=0$. Letting $\theta=0$ and computing as in the previous lemma, it suffices to bound, for $r$ small,
\[
\int\limits\limits_0^{1}\int\limits\limits_{0}^{1}\frac{r^k\phi (r+\phi)^{-k-\alpha} }{y^2+\phi^2}dyd\phi=\int\limits\limits_0^{1} r^k  (r+\phi)^{-k-\alpha}\arctan\left(\frac{1}{\phi}\right)d\phi\leq \frac{C}{1-\alpha}.
\]
For $\frac{(r+\theta)^{k-1}}{1+(r+\theta)^{k-1}} \frac{\partial_\theta^k u^r}{r}$ we see that we need to estimate 
\[
\frac{(r+\theta)^{k+\alpha-1}}{1+(r+\theta)^{k+\alpha-1}}\int\limits\limits_0^{+\infty}\int\limits\limits_{0}^{\frac{\pi}{m}}\frac{\sin\left(m\phi\right)  (r\rho+\left|\phi\right|)^{-k-\alpha}\rho^{-m}}{1-2\rho^{-m} \cos\left(m\phi\right)+\rho^{-2m}}\rho d\phi d\rho,
\]
As above it suffices to consider 
\[
\int\limits\limits_0^{1}\int\limits\limits_{0}^{1}\frac{r^{k+\alpha-1}\phi (r+\phi)^{-k-\alpha} }{y^2+\phi^2}dyd\phi,
\]
which again gives the desired result.
    \end{proof}

    \begin{lemma}\label{lem:est drv utheta}
    For $0\leq \alpha<1$  and $k\in \mathbb{N}$, $k\geq 1$ there exists $C\in \mathbb{R}^*_+$ such that for all $\omega \in \mathscr{S}(\mathbb{R}^2)$
    \[
    \left\vert \frac{(r+\theta)^{k-1}}{1+(r+\theta)^{k-1}}\partial^{k}_r u^\theta(r,\theta)\right\vert \\ \leq  \frac{C}{(1-\alpha)^2} \left(\left\Vert (r+\theta)^{k+\alpha}\partial^{k}_r \omega \right\Vert_{L^\infty}+k\left\Vert (r+\theta)^{k-1+\alpha}\partial^{k-1}_r \omega \right\Vert_{L^\infty}\right),\]
    and 
    \[ \left\vert \frac{(r+\theta)^{k-1}}{1+(r+\theta)^{k-1}} \frac{\partial_\theta^k u^\theta}{r}(r,\theta)\right\vert\leq \frac{C}{1-\alpha}  \left(\left\Vert (r+\theta)^{k+\alpha}\partial_r\partial^{k-1}_\theta \omega \right\Vert_{L^\infty}+\left\Vert (r+\theta)^{k-1+\alpha}\partial^{k-1}_\theta \omega \right\Vert_{L^\infty}\right).
    \]
\end{lemma}
\begin{proof}
    The estimates on $\frac{1}{r}\partial_\theta u^\theta$ are obtained by incompressibility from those on $\partial_r u^r$ as indeed we have 
    \[
\mbox{div}(u)=\partial_x (u_x)+\partial_y(u_y)=\partial_r u_r+\frac{u_r}{r}+\frac{\partial_\theta u_\theta}{r}=\frac{1}{r}\left(\partial_r\left(ru_r\right)+\partial_\theta u_\theta\right)=0.
\]
Starting from
    \begin{equation*}
    U^\theta_1=r\frac{m^2}{2\pi }\int\limits\limits_0^{+\infty}\int\limits\limits_{-\frac{\pi}{m}}^{\frac{\pi}{m}}\frac{\omega(r\rho,\phi)\left(\rho^{-m}-\cos\left(m(\phi-\theta)\right)\right)}{1-2\rho^{-m} \cos\left(m(\phi-\theta)\right)+\rho^{-2m}}\rho^{-m}\rho d\phi d\rho,
\end{equation*}
and for $\frac{(r+\theta)^{k-1}}{1+(r+\theta)^{k-1}}\partial_r^{k} U^\theta_1$ we need to estimate 
\[
r\frac{(r+\theta)^{k-1}}{1+(r+\theta)^{k-1}}\int\limits\limits_0^{+\infty}\int\limits\limits_{-\frac{\pi}{m}}^{\frac{\pi}{m}}\frac{\rho ^k(r\rho+\left|\phi\right|)^{-k-\alpha}\left(\rho^{-m}-\cos\left(m(\phi-\theta)\right)\right)}{1-2\rho^{-m} \cos\left(m(\phi-\theta)\right)+\rho^{-2m}}\rho^{-m}\rho d\phi d\rho.
\]
Again $m\geq 3$ gives integrability at infinity and the singularity in the integral occurs for $\rho=1$ and $\phi=\theta$ and is most severe for $r=\theta=0$. We take $\theta=0$ and after standard computations, it suffices to consider
\[
\int\limits\limits_0^{1}\int\limits\limits_{0}^{1}\frac{r^k(y+\phi^{2}) (r+\phi)^{-k-\alpha} }{y^2+\phi^2}dyd\phi\leq \frac{C}{(1-\alpha)^2}.
\]
Finally $U^\theta_2$ is treated analogously
\end{proof}

\begin{lemma}\label{lem:est drv utheta/r}
    For $0\leq \alpha<1$  and $k\in \mathbb{N}$, $k\geq 1$ there exists $C\in \mathbb{R}^*_+$ such that for all $\omega \in \mathscr{S}(\mathbb{R}^2)$
    \[
    \left\vert \frac{(r+\theta)^{k-1}}{1+(r+\theta)^{k-1}} \partial^k_r \left(\frac{u^\theta(r,\theta)}{r}\right)\right\vert\\ \leq \frac{C}{1-\alpha}\left(\left\Vert (r+\theta)^{k+1+\alpha}\partial^{k+1}_r \omega \right\Vert_{L^\infty}+\left\Vert (r+\theta)^{k+\alpha}\partial^{k}_r \omega \right\Vert_{L^\infty}\right).
    \]
\end{lemma}
\begin{proof}
For this we need to use the cancellation between $U^\theta_1$ and $U^\theta_2$ near $s=r$. Starting from
\[
   U^\theta_1=r\frac{m^2}{2\pi }\int\limits\limits_1^{+\infty}\int\limits\limits_{-\frac{\pi}{m}}^{\frac{\pi}{m}}\frac{\omega(r\rho,\phi)\left(\rho^{-m}-\cos\left(m(\phi-\theta)\right)\right)}{1-2\rho^{-m} \cos\left(m(\phi-\theta)\right)+\rho^{-2m}}\rho^{-m}\rho d\phi d\rho,
\]
\[
    U^\theta_2=r\frac{m^2}{2\pi }\int\limits\limits_0^{1}\int\limits\limits_{-\frac{\pi}{m}}^{\frac{\pi}{m}}\frac{\omega(r\rho,\phi)\left(\cos\left(m(\phi-\theta)\right)-\rho^{m}\right)}{1-2\rho^{m} \cos\left(m(\phi-\theta)\right)+\rho^{2m}}\rho^{m}\rho d\phi d\rho,
\]
changing variables
\[
   U^\theta_2(\omega)=r\frac{m^2}{2\pi }\int\limits\limits_1^{+\infty}\int\limits\limits_{-\frac{\pi}{m}}^{\frac{\pi}{m}}\frac{\omega(r\rho,\phi)\left(\cos\left(m(\phi-\theta)-\rho^{-m}\right)\right)}{1-2\rho^{-m} \cos\left(m(\phi-\theta)\right)+\rho^{-2m}}\rho^{-m}\rho^{-3}d\phi d\rho,
\]
thus 
\[
\frac{2\pi }{m^2}\frac{u^\theta}{r}=\int\limits\limits_{1}^{+\infty}\int\limits\limits_{-\frac{\pi}{m}}^{\frac{\pi}{m}}\frac{\left(\rho^{-m}-\cos\left(m(\phi-\theta)\right)\right)}{1-2\rho^{-m} \cos\left(m(\phi-\theta)\right)+\rho^{-2m}}\rho^{-m}\left(\omega\left(r\rho,\phi\right)\rho-\omega\left(\frac{r}{\rho},\phi\right)\rho^{-3}\right)d\phi d\rho.
\]
Using odd-odd symmetry
\[
\frac{2\pi }{m^2}\frac{u^\theta}{r}=\int\limits\limits_{1}^{+\infty}\int\limits\limits_{0}^{\frac{\pi}{m}}\frac{(\rho^{-2m}-1)\sin(m\phi)\sin(m\theta)\left(\omega\left(r\rho,\phi\right)\rho-\omega\left(\frac{r}{\rho},\phi\right)\rho^{-3}\right)}{\left(1-2\rho^{-m} \cos\left(m(\phi-\theta)\right)+\rho^{-2m}\right)\left(1-2\rho^{-m} \cos\left(m(\phi+\theta)\right)+\rho^{-2m}\right)}\rho^{-m}d\phi d\rho.
\]
Taking $\partial_r^k$ derivatives, after using the fact that $m\geq 3$ to insure integrability in the far-field,  we are brought to estimate 
\[
\frac{(r+\theta)^{k-1}}{1+(r+\theta)^{k-1}} \int\limits\limits_{1}^{2}\int\limits\limits_{0}^{\frac{\pi}{m}}\frac{(\rho^{-2m}-1)\sin(m\phi)\sin(m\theta)\left(\partial_r^k\omega\left(r\rho,\phi\right)\rho^{k+1}-\partial_r^k\omega\left(\frac{r}{\rho},\phi\right)\rho^{-3+k}\right)}{\left(1-2\rho^{-m} \cos\left(m(\phi-\theta)\right)+\rho^{-2m}\right)\left(1-2\rho^{-m} \cos\left(m(\phi+\theta)\right)+\rho^{-2m}\right)}\rho^{-m}d\phi d\rho
\]
Thus it suffices to compute, for $r$ small,
\[
\int^1_{0}\int^{\theta+\frac{1}{m}}_{\theta-\frac{1}{m}} \frac{y^2\theta \phi (r+\theta)^{k-1}r(r+|\phi|)^{-k-1-\alpha} }{(y^2+(\phi-\theta)^2)(y^2+(\phi+\theta)^2)}d\phi dy\leq \frac{C}{1-\alpha}.
\]
\end{proof}

Finally we turn to pull-back estimates needed for the proof of uniqueness.
\begin{corollary}\label{cor:diff pull back}
   Consider an $\omega \in \mathscr{S}(\mathbb{R}^2)$, $0\leq \alpha<1$ and two volume preserving diffeomorphisms $\Phi$ and $\tilde{\Phi}$ of $\mathbb{R}^2$ that are odd symmetric in $\theta$. Suppose moreover that for an $0<\epsilon\ll1$ that we have the estimate 
    \[1-\epsilon \leq \frac{\left\vert \Phi (t,r,\theta)\right\vert}{r},\frac{\left\vert \tilde{\Phi} (t,r,\theta)\right\vert}{r}\leq 1+\epsilon \text{ and } 1-\epsilon \leq \frac{ \mbox{arg}(\Phi)(t,r,\theta)}{\theta},\frac{ \mbox{arg}(\tilde{\Phi})(t,r,\theta)}{\theta}\leq 1+\epsilon.\]
   Then defining 
   \[
u(r,\theta)=\left(\nabla^{\perp}\Delta^{-1}\left(\omega\circ \Phi^{-1}\right)\right)\circ \Phi(r,\theta) \text{ and } \tilde{u}(r,\theta)=\left(\nabla^{\perp}\Delta^{-1}\left(\omega\circ \tilde{\Phi}^{-1}\right)\right)\circ \tilde{\Phi}(r,\theta),
   \]
   and 
   \[   f(t)=\sum^{2}_{k=0}\left\Vert(r+\theta)^{k+\alpha}\partial^k_r \omega  \right\Vert_{L^\infty\left(\mathbb{R}_+ \times\left(0,\frac{\pi}{m}\right)\right)}+\left\Vert(r+\theta)^{k+\alpha}\partial^k_\theta \omega  \right\Vert_{L^\infty\left(\mathbb{R}_+ \times\left(0,\frac{\pi}{m}\right)\right)},
   \]
 we have the $L^\infty$ based difference estimates
   \begin{multline*}
   \left\Vert u^r-\tilde{u}^r\right\Vert_{L^\infty} \leq
        \frac{C}{1-\alpha}\left(\left\Vert(r+\theta)^{\alpha}\partial_r\omega \right\Vert_{L^\infty}+\left\Vert(r+\theta)^{1+\alpha}\partial_r\omega \right\Vert_{L^\infty}+ \left\Vert(r+\theta)^{1+\alpha}\partial_\theta\omega \right\Vert_{L^\infty}\right)\\
        \times \sup_{(r,\theta)\in \mathbb{R}^+\times \left(0,\frac{\pi}{m}\right)} \left| \left\vert \Phi\right\vert-\left\vert \tilde{\Phi}\right\vert\right|+ r\left(1+(r+\theta)^{-\alpha}\right) \left\vert \mbox{arg}(\Phi)-\mbox{arg}(\tilde{\Phi}) \right\vert,
        \end{multline*}
\[
\left\Vert u^\theta-\tilde{u}^\theta\right\Vert_{L^\infty} \leq  \frac{C}{(1-\alpha)^2}f(t)  \sup_{(r,\theta)\in \mathbb{R}^+\times  \left(0,\frac{\pi}{m}\right)} \left| \left\vert \Phi\right\vert-\left\vert \tilde{\Phi}\right\vert\right|+ r \left\vert \mbox{arg}(\Phi)-\mbox{arg}(\tilde{\Phi}) \right\vert.
\]
We also have the weighted $L^1$ based estimates
\begin{multline*}
\int\limits_{\mathbb{R}_+\times \left(0,\frac{\pi}{m}\right)} \frac{\left\vert u^r-\tilde{u}^r\right\vert +\left(1+(r+\theta)^{-\alpha}\right)\left|u^\theta-\tilde{u}^\theta\right| }{r} \frac{\left\vert \omega_0(r,\theta)\right\vert }{1+r^3}rdrd\theta
\\ \leq \frac{C}{(1-\alpha)^2}f(t)
\int\limits_{\mathbb{R}_+\times \left(0,\frac{\pi}{m}\right)} \left(\frac{\left| \left\vert \Phi\right\vert-\left\vert \tilde{\Phi}\right\vert\right|}{r}+ \left(1+(r+\theta)^{-\alpha}\right) \left\vert \mbox{arg}(\Phi)-\mbox{arg}(\tilde{\Phi}) \right\vert\right)\frac{\left\vert \omega_0(r,\theta)\right\vert }{1+r^3}rdrd\theta.
\end{multline*}
\end{corollary}
\begin{proof}
    The proof follows analogously to the proof of the previous lemma by noting that 
    \[
    u(x,t)=\frac{1}{2\pi}\int_{\mathbb{R}^2}\frac{(\Phi(x)-\Phi(y))^\perp}{|\Phi(x)-\Phi(y)|^2}\omega(y,t)dy.
    \]
   We will sketch the ideas for $u^r-\tilde{u}^r$ as all of the other estimates can be treated analogously. First writing 
    \[
    u^r-\tilde{u}^r=(1)+(2),
    \]
    with 
    \[  (1)=\left(\left(\nabla^{\perp}\Delta^{-1}\left(\omega\circ \Phi^{-1}\right)\right)\circ \Phi(r,\theta) -\left(\nabla^{\perp}\Delta^{-1}\left(\omega\circ \Phi^{-1}\right)\right)\circ \tilde{\Phi}(r,\theta)\right)^r
    \]
    and 
    \[
    (2)=\left(\left(\nabla^{\perp}\Delta^{-1}\left(\omega\circ \Phi^{-1}-\omega\circ \tilde{\Phi}^{-1}\right)\right)\circ \tilde{\Phi}(r,\theta)\right)^r.
    \]
    The first term is immediately treated by lemmas above. For the second term, the $L^\infty$ based estimate does not see the exterior composition. Subsequently by incompressibility as well as the bound on the flow the $L^1$ estimates can safely drop the exterior composition.
    
 \noindent For the $L^\infty$ based estimate starting from \ref{eq:u r kernel polar} 
 the term to estimate is the difference between
    \[ (1)=r\frac{m^2}{2\pi }\int\limits\limits_0^{+\infty}\int\limits\limits_{0}^{\frac{\pi}{m}}\frac{\sin\left(m(\mbox{arg}(\Phi)(s,\phi)-\theta)\right)\omega(s,\phi)\left(\frac{ \left\vert \Phi\right\vert(s,\phi)}{r}\right)^{-m}}{1-2\left(\frac{\left\vert \Phi\right\vert(s,\phi)}{r}\right)^{-m} \cos\left(m(\mbox{arg}(\Phi)-\theta)\right)+\left(\frac{\left\vert \Phi\right\vert(s,\phi)}{r}\right)^{-2m}}sd\phi ds,
    \] and 
\[ (2)=r\frac{m^2}{2\pi }\int\limits\limits_0^{+\infty}\int\limits\limits_{0}^{\frac{\pi}{m}}\frac{\sin\left(m(\mbox{arg}(\tilde{\Phi})(s,\phi)-\theta)\right)\omega(s,\phi)\left(\frac{ \left\vert \Phi\right\vert(s,\phi)}{r}\right)^{-m}}{1-2\left(\frac{\left\vert \tilde{\Phi}\right\vert(s,\phi)}{r}\right)^{-m} \cos\left(m(\mbox{arg}(\tilde{\Phi})-\theta)\right)+\left(\frac{\left\vert \tilde{\Phi}\right\vert(s,\phi)}{r}\right)^{-2m}}sd\phi ds.
    \]
    Without loss of generality we work with $\theta=0$ and rewrite the integrals as 
    \[ (1)=r\frac{m^2}{2\pi }\int\limits\limits_0^{+\infty}\int\limits\limits_{0}^{\frac{\pi}{m}}\frac{\sin\left(m\cdot\mbox{arg}(\Phi)(s,\phi)\right)\left(\omega(s,\phi)-\omega(r,0)\right)\left(\frac{ \left\vert \Phi\right\vert(s,\phi)}{r}\right)^{-m}}{1-2\left(\frac{\left\vert \Phi\right\vert(s,\phi)}{r}\right)^{-m} \cos\left(m \cdot\mbox{arg}(\Phi)(s,\phi)\right)+\left(\frac{\left\vert \Phi\right\vert(s,\phi)}{r}\right)^{-2m}}sd\phi ds,
    \] and 
\[ (2)=r\frac{m^2}{2\pi }\int\limits\limits_0^{+\infty}\int\limits\limits_{0}^{\frac{\pi}{m}}\frac{\sin\left(m\cdot\mbox{arg}\tilde{\Phi}(s,\phi)\right)\left(\omega(s,\phi)-\omega(r,0)\right)\left(\frac{ \left\vert \tilde{\Phi}\right\vert(s,\phi)}{r}\right)^{-m}}{1-2\left(\frac{\left\vert \tilde{\Phi}\right\vert(s,\phi)}{r}\right)^{-m} \cos\left(m\cdot\mbox{arg}\tilde{\Phi}(s,\phi)\right)+\left(\frac{\left\vert \tilde{\Phi}\right\vert(s,\phi)}{r}\right)^{-2m}}sd\phi ds.
    \]
    We introduce
    \[
(aux)_1=r\frac{m^2}{2\pi }\int\limits\limits_0^{+\infty}\int\limits\limits_{0}^{\frac{\pi}{m}}\frac{\sin\left(m\cdot\mbox{arg}(\Phi)(s,\phi)\right)\left(\omega(s,\phi)-\omega(r,0)\right)\left(\frac{ \left\vert \tilde{\Phi}\right\vert(s,\phi)}{r}\right)^{-m}}{1-2\left(\frac{\left\vert \Phi\right\vert(s,\phi)}{r}\right)^{-m} \cos\left(m \cdot\mbox{arg}(\Phi)(s,\phi)\right)+\left(\frac{\left\vert \Phi\right\vert(s,\phi)}{r}\right)^{-2m}}sd\phi ds,
\]
\[
(aux)_2=r\frac{m^2}{2\pi }\int\limits\limits_0^{+\infty}\int\limits\limits_{0}^{\frac{\pi}{m}}\frac{\sin\left(m\cdot\mbox{arg}(\Phi)(s,\phi)\right)\left(\omega(s,\phi)-\omega(r,0)\right)\left(\frac{ \left\vert \tilde{\Phi}\right\vert(s,\phi)}{r}\right)^{-m}}{1-2\left(\frac{\left\vert \tilde{\Phi}\right\vert(s,\phi)}{r}\right)^{-m} \cos\left(m \cdot\mbox{arg}(\Phi)(s,\phi)\right)+\left(\frac{\left\vert \tilde{\Phi}\right\vert(s,\phi)}{r}\right)^{-2m}}sd\phi ds,
\]
and 
\[
(aux)_3=r\frac{m^2}{2\pi }\int\limits\limits_0^{+\infty}\int\limits\limits_{0}^{\frac{\pi}{m}}\frac{\sin\left(m\cdot\mbox{arg}(\tilde{\Phi})(s,\phi)\right)\left(\omega(s,\phi)-\omega(r,0)\right)\left(\frac{ \left\vert \tilde{\Phi}\right\vert(s,\phi)}{r}\right)^{-m}}{1-2\left(\frac{\left\vert \tilde{\Phi}\right\vert(s,\phi)}{r}\right)^{-m} \cos\left(m \cdot\mbox{arg}(\Phi)(s,\phi)\right)+\left(\frac{\left\vert \tilde{\Phi}\right\vert(s,\phi)}{r}\right)^{-2m}}sd\phi ds,
\]
From an analogous estimate to Lemma \ref{lem:lin u sing} the difference $(1)-(aux)_1$ yields immediately an upper bound of the form 
\[
\left\vert (1)-(aux)_1 \right\vert \leq \frac{C}{1-\alpha} \left\Vert(r+\theta)^{\alpha}\partial_r\omega \right\Vert_{L^\infty} \sup_{(r,\theta)\in \mathbb{R}^+\times \left(0,\frac{\pi}{m}\right)} \left| \left\vert \Phi\right\vert-\left\vert \tilde{\Phi}\right\vert\right|.
\]
Next for $(aux)_1-(aux)_2$ we see that in the far-field we have an analogous estimate to the one above due to $m\geq 3$ and that the most singular part to estimate is for $s\approx r$ and $\phi\approx 0$. After standard computation and using the hypothesis on $\Phi$ and $\tilde{\Phi}$ we get a factor of $\sup_{(r,\theta)\in \mathbb{R}^+\times \left(0,\frac{\pi}{m}\right)} \left| \left\vert \Phi\right\vert-\left\vert \tilde{\Phi}\right\vert\right|$ multiplied by an integral of the form
\[
r\int^{2r}_{\frac{r}{2}}\int_0^1\frac{\phi\left(\left|s-r\right|+\phi)\right)(s+\phi)^{-1-\alpha}}{\left(\left|1-\frac{s}{r}\right|^2+\phi^2\right)^{\frac{3}{2}}}dsd\phi\leq \frac{C}{1-\alpha}.
\]
Next for $(aux)_2-(aux)_3$ the most severe singularity is again for $s\approx r$ and $\phi\approx 0$. We get a factor of $\sup_{(r,\theta)\in \mathbb{R}^+\times \left(0,\frac{\pi}{m}\right)}r\left(1+(r+\theta)^{-\alpha}\right) \left\vert \mbox{arg}(\Phi)-\mbox{arg}(\tilde{\Phi})\right\vert$ multiplied by an integral of the form
\[
r^\alpha\int^{2r}_{\frac{r}{2}}\int_0^1\frac{\left(\left|s-r\right|+\phi)\right)(s+\phi)^{-1-\alpha}}{\left|1-\frac{s}{r}\right|^2+\phi^2}dsd\phi\leq \frac{C}{1-\alpha}.
\]
The estimate for $(aux)_3-(2)$ is treated analogously to $(aux)_1-(aux)_2$. Finally the $L^1$ based estimates are obtained in the same fashion.
\end{proof}

\appendix

\section{Kernel computation}
\begin{lemma}\label{lem:inf sum comp}
We have the identities for $\lambda<1$
\[
\sum^{+\infty}_{k=1}  \lambda^{k}\sin(mk\phi)\cos(mk\theta)=\frac{\lambda}{2}\left(\frac{\sin\left(m(\phi+\theta)\right)}{1-2\lambda \cos\left(m(\phi+\theta)\right)+\lambda^{2}}+\frac{\sin\left(m(\phi-\theta)\right)}{1-2\lambda \cos\left(m(\phi-\theta)\right)+\lambda^{2}}\right),
\]
thus,
\[
\sum^{+\infty}_{k=1}  \lambda^{k}\frac{\sin(mk\phi)\sin(mk\theta)}{mk}=\frac{1}{4}\ln\left(\frac{1-2\lambda \cos\left(m(\phi+\theta)\right)+\lambda^{2}}{1-2\lambda \cos\left(m(\phi-\theta)\right)+\lambda^{2}}\right).
\]
\end{lemma}
\begin{proof} Compute
\begin{align*}
&\sum^{+\infty}_{k=1}  \lambda^{k}\sin(mk\phi)\cos(mk\theta)=\frac{1}{4i}\sum^{+\infty}_{k=1}  \lambda^{k} \left(e^{imk\phi}-e^{-imk\phi}\right)\left(e^{imk\theta}+e^{-imk\theta}\right)\\
&=\frac{1}{4i}\sum^{+\infty}_{k=1}  \lambda^{k} \left(e^{imk(\phi+\theta)}-e^{imk(\theta-\phi)}+e^{imk(\phi-\theta)}-e^{-imk(\phi+\theta)}\right)\\
&=\frac{1}{4i}\left(\frac{1}{1-\lambda e^{im(\phi+\theta)}}-\frac{1}{1-\lambda e^{-im(\phi+\theta)}}+\frac{1}{1-\lambda e^{im(\phi-\theta)}}-\frac{1}{1-\lambda e^{-im(\phi-\theta)}}\right)\\
&=\frac{1}{2}\left(\frac{\lambda\sin\left(m(\phi+\theta)\right)}{1-2\lambda \cos\left(m(\phi+\theta)\right)+\lambda^{2}}+\frac{\lambda\sin\left(m(\phi-\theta)\right)}{1-2\lambda \cos\left(m(\phi-\theta)\right)+\lambda^{2}}\right).
    \end{align*}

\end{proof}

\bibliographystyle{plain}
\bibliography{non-uniqueness-Lp}
\end{document}